\documentclass[11pt, twoside]{amsart}
\usepackage{shadethm,framed}
\usepackage[utf8]{inputenc}
\usepackage[english]{babel}
\usepackage[papersize={21cm,29.7cm},left=3.3cm,right=3.3cm,top=3.17cm,bottom=3.160cm]{geometry}
\usepackage{enumerate}
\usepackage{listings}
\usepackage{amsthm, amsfonts, amssymb, amsmath}
\usepackage{graphicx}
\usepackage{xcolor}
\usepackage{color}
\usepackage{framed}
\usepackage{wallpaper}
\usepackage{tikz,tikz-cd}
\usepackage{pgf,tikz,pgfplots}
\usepackage{tcolorbox}
\usepackage{scalefnt}
\usepackage{mdframed}
\usepackage{booktabs}
\usepackage{caption}
\usepackage{makecell}
\usepackage{float}
\usepackage{thmtools}
\usepackage{thm-restate}
\usepackage[all,2cell]{xy}
\usepackage{todonotes}
\usepackage{csquotes}
\usepackage[T1]{fontenc}

\usepackage{hyperref}

\usetikzlibrary{patterns, positioning, arrows,decorations.markings}

\usepackage{cleveref}
\usepackage{pgfplots}
\pgfplotsset{width=7cm, compat=1.10}
\usepgfplotslibrary{fillbetween}

\newtheorem{lem}{Lemma}[section]
\newtheorem{cor}{Corollary}[lem]
\newtheorem{thm}{Theorem}[lem]
\newtheorem{prop}{Proposition}[lem]

\makeatletter
\renewcommand{\l@subsection}{\@tocline{2}{0pt}{2.5em}{3.5em}{}}
\makeatother

\theoremstyle{definition}
\newtheorem{definition}[lem]{Definition}
\newtheorem{ex}[lem]{Example}

\newtheorem{rmk}[lem]{Remark}
\newtheorem{conj}[lem]{Conjecture}

\newcommand{\cadre}{\Box}
\newcommand{\carre}{\boxplus}
\newcommand{\cross}{+}
\newcommand{\N}{\mathbb N}
\newcommand{\Z}{\mathbb Z}
\newcommand{\Q}{\mathbb Q}

\newcommand{\F}{\mathbb F}
\newcommand{\Sph}{\mathbb S}

\newcommand{\hall}{\mathcal{H}}

\pgfplotsset{compat=1.18}

\title{Hall algebras of graphs and rooted trees}
\author{Lucas Toury}

\begin{document}

\begin{abstract} 
Hall algebras can be associated with a broad class of combinatorial structures through the theory of 2-Segal sets. In this paper, we study the Hall algebras arising from the 2-Segal sets of rooted trees, undirected graphs, and directed graphs. In each case, we establish an analogue of Green’s theorem giving a twisted bialgebra structure to the Hall algebra, describe the primitive elements, and derive a presentation by generators and relations. As an application, we realize the Hall algebra of an undirected graph as the cohomology ring of a topological space. In the case of directed graphs, we introduce the Hall polynomial, defined as the Poincaré polynomial of the space of primitive elements. The Hall polynomial is an invariant of the underlying undirected graph which we show to be closely related to the Tutte polynomial. We furthermore give examples of graphs with equal Tutte polynomials but distinct Hall polynomials showing thus that the Hall polynomial encodes different information from that of the Tutte polynomial.
\end{abstract}

\maketitle
\tableofcontents

\section{Introduction}
Hall algebras of small abelian categories were introduced by Ringel, in particular for the categories of representations of a quiver over a finite field. The multiplication encodes all possible extensions between objects. Interest in Hall algebras grew considerably when Ringel highlighted a deep connection between the Kac--Moody algebra associated with a quiver and the Hall algebra of the category of representations of this quiver \cite{Ringel1990}. A few years later, Green introduced a coproduct for Hall algebras \cite{greenhereditaryalgebra1995} that dually encodes the ways an object can appear as the middle term of a short exact sequence. Green showed that the Hall algebra constructed from a hereditary abelian category possesses a twisted bialgebra structure, thereby extending Ringel's result.

More recently, Dyckerhoff and Kapranov, and independently, Gálvez-Carrillo, Kock, and Tonks, proposed an object that allows a generalization of the construction of Hall algebras. This object is known as a 2-Segal space in the former case \cite{DyckerhoffKapranov2019}, or a decomposition space in the latter \cite{zbMATH06894552}.

To a 2-Segal space satisfying suitable finiteness conditions, we can associate a Hall algebra. In this article, we begin a systematic study of Hall algebras associated with various types of graphs: labelled planar rooted trees, labelled undirected graphs, and labelled directed graphs. The cases of trees and undirected graphs have already been introduced in \cite{bergner20172segalsetswaldhausenconstruction} and \cite{bergner20252segalsetscutsrooted}.

\smallbreak

A Hall algebra is equipped with a product as well as a coproduct. These two structures are adjoint with respect to a certain scalar product. As in the context of quivers, this does not yield an ordinary bialgebra. One of our motivations was to find a compatibility between these two structures. In each context, we prove that the Hall algebra has a twisted bialgebra structure (see \cref{thmbialgebraLRT}, \cref{thmbialgebraNOL}, and \cref{thmbialgebraOL}), that is, it is an algebra and a coalgebra such that
$$\Delta(ab)=\Delta(a)\cdot \Delta(b),$$
where $\Delta$ is the coproduct. The name ``twisted'' comes from the fact that the multiplication is twisted by a coefficient that we discuss in \cref{defusualtwist}:
$$(a_1\otimes a_2)\cdot (b_1\otimes b_2)=\phi(a_1,a_2,b_1,b_2)(a_1b_1\otimes a_2b_2).$$

In each of the three contexts, we give an explicit description of the space of primitive elements, i.e.\ elements $x$ such that $\Delta(x)=1\otimes x + x \otimes 1$. They often reflect combinatorial properties of the underlying object. They generate the Hall algebra and allow us to give a presentation by generators and relations using the twisted bialgebra structure. This overall framework, along with the main results, is summarized in \cref{tabularintro}.

\smallbreak

Since the Hall algebra of an undirected graph is commutative, Bergner, Kuhn, and Zakharevich \cite{bergnerunpublishednote} asked whether there exists a topological space associated with a graph whose cohomology ring recovers the Hall algebra. In its strongest form, they conjecture that
\begin{conj}
For any finite graph $G$, there is an algebra isomorphism
$$\hall(G)\simeq H^*(|X_G|,\mathbf{F}_2),$$
where $\hall(G)$ is the Hall algebra associated with the graph $G$ and $|X_G|$ denotes the geometric realization of the 2-Segal simplicial set $X_G$ associated with $G$.
\end{conj}

We are not able to answer this question, but as a partial answer, we construct a topological space as a polyhedral product whose cohomology with integral coefficients is isomorphic to the Hall algebra (\cref{thmpolyhedralproductNOL}).

\smallbreak

For directed graphs, we introduce the Poincaré polynomial of the space of primitives, which we call the Hall polynomial. Although it exists in the other two cases, it seems to be of greater interest here. It provides an invariant of the underlying graph that has strong connections with the Tutte polynomial. However, it is not contained in the latter: we give examples of graphs with the same Tutte polynomial but different Hall polynomials. At present, we have not found graphs with the same Hall polynomial but different Tutte polynomials.
\bigbreak
\begin{figure}
\begin{center}
\begin{tabular}{|>{\centering\arraybackslash}m{2.4cm}|>{\centering\arraybackslash}m{3.4cm}|>{\centering\arraybackslash}m{3.4cm}|>{\centering\arraybackslash}m{3.4cm}|}
\hline
Combinatorial structure & Rooted trees (\ref{sectionLRT}) & Non directed \newline graphs (\ref{sectionNOL})& Directed graphs  (\ref{sectionOL})\\
\hline
Basis of the Hall algebra & Admissible subforests & Subgraphs & Subgraphs \\
\hline
\vspace{1cm}
Type of element appearing in the product of $A$ by $B$ & \begin{center}
\begin{tikzpicture}[scale=0.5]
\node[color=blue] at (-1,0) {$B$};
\node[color=red] at (-3,-3) {$A$};
    \node[circle, fill=blue] (A) at (-3,-1) {};
    \node[circle, fill=blue] (B) at (-3,0) {};
    \node[circle, fill=red] (C) at (-2,-2) {};
    \node[circle, fill=red] (D) at (-1,-3) {};
    \node[circle, fill=red] (E) at (-1,-2) {};
    \node[circle, fill=red] (F) at (-1,-1) {};
    \node[circle, fill=blue] (G) at (0,-1) {};
    \node[circle, fill=red] (H) at (-2,-3) {};

    \draw (A)--(B);
    \draw[color=green] (A)--(C);
    \draw (C)--(H);
    \draw (C) -- (F);
    \draw (D) --(E);
    \draw[color=green] (E) -- (G);
\end{tikzpicture}\end{center}&
\begin{center}
\begin{tikzpicture}
    \node[circle, draw] (A) at (-2,0) {$A$};
    \node[circle, draw] (B) at (0,0) {$B$};

    \draw[thick] (A) -- (B);
    \draw[thick, bend left=30] (A) to (B);
    \draw[thick, bend right=30] (A) to (B);
\end{tikzpicture}\end{center}&\begin{center} \begin{tikzpicture}
    \node[circle, draw] (A) at (-2,0) {$A$};
    \node[circle, draw] (B) at (0,0) {$B$};

    \draw[->, thick] (A) -- (B);
    \draw[->, thick, bend left=30] (A) to (B);
    \draw[->, thick, bend right=30] (A) to (B);
\end{tikzpicture}\end{center} \\  
&\scriptsize{$B$ is just above $A$ in the common branches}& &\\ \hline Product & Semigroup & Commutative & Non-commutative \\
\hline
Primitives are in bijection with & Vertices (\ref{propprimitivesLRT}) & Connected subgraphs (\ref{propprimitivesNOL}) & Strong orientation \newline of induced subgraphs (\ref{propprimitivesOL})\\
\hline
Relations between primitives & • $x_ix_j$ if $i$ is in a common branch with $j$ and is not the vertex just above $j$ \smallbreak • $x_ix_j-x_jx_i$ if $i$ and $j$ are not in a common branch (\ref{thmpresentationLRT}) &• $x_Px_K-x_Kx_P$ \smallbreak • $x_Px_K$ if $P\cap K\neq \emptyset$ (\ref{thmpresentationNOL})&• $x_Px_K-x_Kx_P$ if there is not edge between $P$ and $K$ \smallbreak • $x_{P_1}...x_{P_k}$ if there exists $i\neq j$ s.t. $P_i\cap P_j\neq \emptyset$ (\ref{thmpresentationOL})\\
\hline
\end{tabular}
\caption{Summary of Hall algebras and presentations} \label{tabularintro}
\end{center}
\end{figure}

\section*{Acknowledgments}

I would like to thank my advisor, Dragoș Frățilă, for his guidance and for introducing me to the vast world of Hall algebras, which I feel I have only just begun to explore. I would like to thank Jean-Sébastien Sereni for all our interesting discussions on graphs and combinatorics, as well as Vladimir Dotsenko, who wisely pointed out to me that the polyhedral product was the right approach for my problem.

\section{2-Segal sets and Hall algebras}\label{sectiontheorique}
In this section we collect some theoretical facts about 2-Segal sets and their Hall algebras. See \cite[section 2.3]{DyckerhoffKapranov2019} and \cite[Intro, Section 1]{bergner20172segalsetswaldhausenconstruction} for more details.  
\begin{definition}{\cite{DyckerhoffKapranov2019}}\label{def2segal}
    A simplicial set $X$ is 2-Segal if for every $n\geq 3$ and $0\leq i < j \leq n$, the map 
    $$X_n\rightarrow X_{\{0,1,...,i,j,j+1,...,n\}}\times_{X_{\{i,j\}}}X_{\{i,i+1,...,j\}}$$
    induced by the inclusions $\{0,1...,i,j,j+1,...,n\},\{i,i+1,...j\}\subset [n]$ is a bijection.
    \smallbreak
    We say that $X$ is reduced if $X_0$ is a singleton. In this case we will denote by $*$ the image of the element of $X_0$ in $X_1$ by the degeneracy map $s_0$.
\end{definition}

\begin{definition}\label{producthallalgebra}
    For $X$ a reduced 2-Segal set, we define the Hall algebra associated to $X$ as follow :
    \newline The underlying vector space over $\Q$ is spanned by the elements of $X_1$. Let $a,b  \in X_1$, the multiplication is defined by
    $$ a b=\sum_{h\in X_1}c_{a,b}^hh,$$
    where $c_{a,b}^h=|\{x\in X_2~|~d_0(x)=b,d_2(X)=a,d_1(X)=h\}|$. The unit is given by $*$. We will denote this algebra by $\hall(X)$.
\end{definition}
\begin{rmk}
\begin{itemize}
\item Usually, to define a Hall algebra, certain finiteness conditions are required. In all of our studies, since $X_1$ and $X_2$ are finite, these conditions are automatically satisfied, and we will not mention them. 
\item The associativity of this algebra is due to the 2-Segal condition (see \cite[Prop 3.4.6]{DyckerhoffKapranov2019}). 
\end{itemize}
    
\end{rmk}
This algebra is also naturally equiped with a coassociative coproduct.
\begin{definition}\label{coproducthallalgebra}
     Let $h\in X_1$, we define a coproduct and a counit on $\hall(X)$ by
    $$\Delta(h)=\sum_{(a,b)\in X_1(G)\times X_1(G)}c_{a,b}^ha\otimes b,$$
    and
    \begin{equation*} \epsilon(H) = \left\{
    \begin{array}{ll}
        1 ~~\mbox{ if }h=*,\\
        0 ~~\mbox{ otherwise. }
    \end{array}
\right.\end{equation*}
\end{definition}
\begin{definition}
  We define a positive-definite scalar product by
  \begin{equation*} \langle a,b\rangle = \left\{
    \begin{array}{ll}
        1 ~~\mbox{ if }a=b,\\
        0 ~~\mbox{ otherwise. }
    \end{array}
\right.\end{equation*}
The coproduct and product are adjoint  with respect to the scalar product :
\newline For $a,b,c\in X_1$
\begin{equation*}
    \langle a b,c\rangle=\langle a\otimes b,\Delta(c)\rangle,
\end{equation*}
where $\langle a\otimes b, c\otimes d\rangle $ is defined to be $\langle a,c\rangle \langle b,d\rangle$.
\end{definition}
\begin{rmk}
    Notice that in formulas of product and coproduct all the coefficients are in $\Z$. So we can also define $\hall(X)_\Z$ the Hall algebra of $X$ with coefficients in $\Z$.
\end{rmk}
\subsection{Twisted bialgebra}
Having a product and a coproduct on $\hall(X)$, it is natural to see wether they give rise to a bialgebra structure. It turns out that it is not the case in general, but sometimes it works by adding a twist like in Green's theorem. What follows is inspired of \cite[Section II.1]{zbMATH00871886}.
\medbreak
In this section we assume that $\mathcal{A}$ is a vector space endowed with an algebra and coalgebra structure. We also make the assumption that $\mathcal{A}$ is $\N I$-graded meaning that there exists a set $I$ such that
$$\mathcal{A}=\bigoplus_{x\in \N I}\mathcal{A}_x$$
as a vector space and that the product and the coproduct respect the grading. 
\begin{definition}
    We consider $\Z \sqcup \{-\infty\}$ as a commutative monoid with the usual addition for elements in $\Z$ and with $-\infty+a=-\infty+-\infty=-\infty$ for all $a\in \Z$.
\end{definition}
\begin{definition}\label{defusualtwist}
For  $\phi$, $\phi'$ two bilinear applications from $\N I\times \N I\rightarrow\Z\sqcup \{-\infty\}$ and $v\in \Q^*$, we define a product on $\mathcal{A}\otimes \mathcal{A}$ by :
\begin{equation}\label{deftwist}
    (a\otimes b)\cdot (c\otimes d)=v^{\phi(w,z)+\phi'(x,y)}(ac\otimes bd)
\end{equation}
for $a\in \mathcal{A}_w$ $b\in \mathcal{A}_x$, $c\in \mathcal{A}_y$ and $d\in \mathcal{A}_z$, where we define $v^{-\infty}$ te be $0$.
\end{definition}
\begin{rmk}
    For readability we will note $\phi(a,b)$ instead of $\phi(x,y)$ for $a\in \mathcal{A}_x$ and $b\in \mathcal{A}_y$.
\end{rmk}
\begin{prop}
    The product (\ref{deftwist}) on $\mathcal{A}\otimes \mathcal{A}$ is associative.
\end{prop}
\begin{proof}
    Let $a_1,a_2,b_1,b_2,c_2,c_2\in \mathcal{A}$, 
    \begin{align*}
        ((a_1\otimes a_2)\cdot& (b_1\otimes b_2))\cdot (c_1\otimes c_2)\\
        =& v^{\phi(a_1,b_2)+\phi'(a_2,b_1)}(a_1b_1\otimes a_2b_2)(c_1\otimes c_2)\\
        =& v^{\phi(a_1,b_2)+\phi'(a_2,b_1)+\phi(a_1b_1,c_2)+\phi'(a_2b_2,c_1)}a_1b_1c_1\otimes a_2b_2c_2\\
        =& v^{\phi(a_1,b_2)+\phi'(a_2,b_1)+\phi(a_1,c_2)+\phi(b_1,c_2)+\phi'(a_2,c_1)+\phi'(b_2,c_1)}a_1b_1c_1\otimes a_2b_2c_2.
    \end{align*}
    where the last equality follow from the respect of the grading by the multiplication and the linearity of $\phi$ and $\phi'$. It is now easy to see that it is the same as $(a_1\otimes a_2)\cdot ((b_1\otimes b_2)\cdot (c_1\otimes c_2))$.
    $\qedhere$
\end{proof}
\begin{definition}\label{deftwistedbialgebra}
    We say that $\mathcal{A}$ is a $(\Q,v,\phi,\phi')$-bialgebra if $\mathcal{A}$ satisfies the relation
    $$\Delta(a b)=\Delta(a)\cdot \Delta(b)$$
    for all $a$, $b\in \mathcal{A}$ with the twisted multiplication (\ref{deftwist}) defined on $\mathcal{A}\otimes \mathcal{A}$.
\end{definition}

\subsection{Primitive elements}
In the following section we assume that $X$ is a reduced 2-Segal set and that $\hall(X)$ is a $\N I$-graded algebra and coalgebra such that with the $\N$ graduation (after forgetting labels of $I$) every degree is finite dimensional. We also assume that $\hall(X)$ is connected i.e. $\hall(X)_0\simeq \Q$ and that for each degree a basis is given by elements of $X_1$. It implies in particular that the degrees are orthogonal to one another. 
\begin{definition}
    We say that an element $x$ of $\hall(X)$ is primitive if $\Delta(x)=1\otimes x + x \otimes 1$. The vector space of primitive elements will be denoted by $Prim(X)$.
\end{definition}
\begin{rmk}
    Since $\hall(X)$ is $\N$-graded, it implies that $Prim(X)$ is $\N$-graded too. Indeed
    $$Prim(X)=\bigoplus_{i=1}^\infty Prim(X)_i,$$
    where $Prim(X)_i:=Prim(X)\cap \hall(X)_i$ for $i\in \N$. 
\end{rmk}
\begin{definition}
    For $i \geq 1$ we define
    $$\hall(X)_{<i}:=\bigoplus_{k=0}^{i-1}\hall(X)_k.$$ 
\end{definition}
The following proposition allows in particular to calculate the graded dimension of the space of primitives
\begin{prop}
   We have $$Prim(X)_i=(\hall(X)_{<i}\times \hall(X)_{<i})^{\perp}\cap \hall(X)_i,$$
   where $\hall(X)_{<i}\times \hall(X)_{<i}$ means sum of products of elements in $\hall(X)_{<i}$.
\end{prop}
\begin{proof}
    If $p\in Prim(X)_i$ then $p\in \hall(X)_i$. Let $a,b\in \hall(X)_{<i}$, then $$\langle p,a b\rangle =\langle\Delta(p),a\otimes b\rangle =\langle p,a\rangle \langle 1,b\rangle +\langle 1,a\rangle \langle p,b\rangle =0$$ because $ \langle p,a\rangle =\langle p,b\rangle =0$ by the grading. So $p\in (\hall(X)_{<i}\times \hall(X)_{<i})^{\perp}$. 
    \smallbreak
    For the converse inclusion, let $p\in (\hall(X)_{<i}\times \hall(X)_{<i})^{\perp}\cap \hall(X)_i$. By definition $$\Delta(p)=1\otimes p+ p\otimes 1 + \sum p_1\otimes p_2$$ with $p_1\otimes p_2\in \hall(X)_{<i}\otimes \hall(X)_{<i}$.
    
    By adjunction $$\langle\Delta(p),p_1\otimes p_2\rangle =\langle p,p_1p_2\rangle =0,$$ so $\Delta(p)=1\otimes p+ p\otimes 1$ and $p$ is primitive.
    $\qedhere$
\end{proof}
\begin{prop}\label{corgenerate}
 The vector space $Prim(X)$ of primitive elements generates $\hall(X)$ as an algebra.
\end{prop}
\begin{proof}
Denote by $B$ the subalgebra of $\hall(X)$ generated by primitive elements. We show that $B^\perp$ is a coideal, indeed for $b\in B^\perp$ and $b_1,b_2$ in $B$ we have 
$$\langle \Delta(b),b_1\otimes b_2 \rangle = \langle b,b_1b_2 \rangle = 0.$$
So $\Delta(B^\perp)\subseteq (B\otimes B)^\perp$. Thanks to the scalar product we have $$(B\otimes B)^\perp \subseteq B^\perp \otimes \hall(X) + \hall(X) \otimes B^\perp$$ and hence $B^\perp$ is a coideal.
We use the following lemma :
\begin{lem}
   Every non-trivial $\N$-graded coideal of $\hall(X)$ contains a primitive element.
\end{lem}
\begin{proof}
Let $I$ be a $\N$-graded non-trivial coideal. Choose the smallest $n\in \N^*$, such that $I_n\neq \hall(X)_n$ (it exists as $I$ is non-trivial). For $a\in I_n$, its coproduct is of the form $\sum p_1\otimes p_2+1\otimes a + a \otimes 1$ with either $p_1$ or $p_2$ in $I$. Because the coproduct respects the grading both $p_1$ and $p_2$ are of degree strictly smaller than $n$, hence $0$ by the minimality assumption. This shows that $a$ is primitive.

$\qedhere$
\end{proof}
It is easy to see that $B^\perp$ is a $\N$-graded coideal. If $B^\perp$ is not $\{0\}$ there exists a primitive element in $p\in B^\perp$ and so $p\in B \cap B^\perp$. As we are dealing with a scalar product we have $p=0$. So $B^\perp=\{0\}$ and as $\hall(X)$ is finite dimensional in every degree we conclude that $B=\hall(X)$.
$\qedhere$
\end{proof}
\begin{rmk}
    In the previous proof, we make extensive use of the fact that each degree is of finite dimension and that the degrees are orthogonal to one another.
\end{rmk}
\begin{definition}
    For $n\in \N$, we define $\Delta^n$ by $\Delta^0:=id$ and $\Delta^{n+1}:=(\Delta\otimes id^{n})\circ \Delta^n$.
\end{definition}
\begin{lem}
   We have the following generalisation of adjunction :
    \begin{equation}\label{generalizedadjunction}
        \langle x_1...x_n, x\rangle =\langle x_1\otimes ... \otimes x_n , \Delta^{n-1}(x)\rangle .
    \end{equation}
\end{lem}
\begin{prop}\label{formulprimitivegeneral}
Assume that $\hall(X)$ is a $(\Q,v,\phi,\phi')$-bialgebra where $\phi(1,a)=\phi(a,1)=\phi'(1,a)=\phi'(a,1)=0$ for all $a \in \hall(X)$.
For $p_1,...,p_k$ be primitive elements of $\hall(X)$ we have 
    \begin{equation*}
        \Delta^n(p_1...p_k)=\sum_{I_0\sqcup...\sqcup I_n=\{1,...,k\}}v^{\sum_{i<j}\phi(\text{Ord(}I_i,I_j))+\phi'(\text{Inv}(I_i,I_j))}\prod_{i\in I_0}p_i\otimes ... \otimes \prod_{i\in I_n}p_i,
    \end{equation*}
    where $\text{Ord}(I_i,I_j):=\{(p_k,p_l)~|~k\in I_i,l\in I_j, i<j, k<l\}$, $\text{Inv}(I_i,I_j):=\{(p_k,p_l)~|~k\in I_i,l\in I_j, i<j, k>l\}$ and $\phi(\{(a_i,b_i)\}_{i=0}^d):=\sum_{i=0}^d\phi(a_i,b_i)$.
\end{prop}
\begin{proof}
This formula looks like the usual one for the iterated coproduct of a product of primitive elements in a usual bialgebra. Now to check the coefficient, just see in the definition of the twist that for $i<j$ if $p_i$ and $p_j$ are not in the same part of the tensor product there is always a twist between these elements. If  $p_i$ appears first the twist is given by $\phi(a,b)$, if it is reverse then it is given by $\phi'(a,b)$. Now use the bilinearity of $\phi$ and $\phi'$ to conclude the proof.
    $\qedhere$
\end{proof}
Now we recall a proposition from \cite{Berenstein_2016} (Lemma 6.14) and complete it with another result.
\begin{prop}\label{proporthogonalityprimitive}
We have the following results about interaction between primitive elements :
    \begin{equation*}
        \langle \sum_{k\geq 2}Prim(X)^k,Prim(X)\oplus \Q\rangle =0
    \end{equation*}
Moreover, let $p_1,...,p_n,q_1,...,q_k$ be primitive elements with $k\geq n$ and such that in the case where $k=n$ the two sets $\{p_1,...,p_n\}$ and $\{q_1,...,q_n\}$ are different. We have
\begin{equation*}
    \langle p_1...p_n,q_1,...,q_k\rangle =0.
\end{equation*}
\end{prop}
\begin{proof}
The first result is the lemma 6.16 from \cite{Berenstein_2016}. For the second part, use \cref{generalizedadjunction} and the \cref{formulprimitivegeneral}.
\begin{align*}
   & \langle p_1...p_n,q_1,...,q_k\rangle  =\langle p_1\otimes ... \otimes p_n, \Delta^{n-1}(q_1...q_k)\rangle  \\
    & = \sum_{I_0\sqcup...\sqcup I_{n-1}=\{1,...,k\}}v^{\sum_{i<j}\phi(\text{Ord(}I_i,I_j))+\phi'(\text{Inv}(I_i,I_j))}\langle p_1,\prod_{i\in I_0}q_i\rangle  ...\langle p_n,\prod_{i\in I_{n-1}}q_i\rangle .
\end{align*}
In the case where $k>n$, there exists at least one $I_j$ such that $|I_j|\geq 2$ in each term of the sum, and we conclude by the first result.

If $k=n$, if there exists an empty $I_j$, by definition $\langle p_j,1\rangle =0$. Otherwise every $I_j$ is of cardinal one but by hypothesis there exists $i_0$ such that $\langle p_i,q_{i_0}\rangle =0$ for every $i$, so every terms are zero.
$\qedhere$
\end{proof}
\begin{prop}\label{proprelationprimitifs}
 A minimal relation among primitive elements, that is, an expression of the form
 $\sum_i \lambda_i~p_{i,1}...p_{i,k_i}=0$
which cannot be decomposed into a sum of relations of strictly smaller length, can occur only when each monomial involves the same multiset of primitive elements, possibly in a different order.
\end{prop}
\begin{proof}
Assume that we have a minimal relation $\sum_i \lambda_i~p_{i,1}...p_{i,k_i}=0$. Take a product of this sum denoted by $p_1...p_k$. Define $$x:=\sum_{i \text{ s.t. }\{p_{i,1},...,p_{i,k_i}\}=\{p_1,...,p_k\}}a_i p_{i,1}...p_{i,k_i}.$$
Now using \cref{proporthogonalityprimitive} we compute
$$0=\langle 0,x\rangle=\langle \sum_i \lambda_i~p_{i,1}...p_{i,k_i},x\rangle =\langle x,x\rangle $$
and so $x=0$. By the minimality of the relation we conclude that $x$ is the whole relation and $x$ is of the desired form.
$\qedhere$
\end{proof}

\section{Rooted trees}\label{sectionLRT}
\subsection{2-Segal sets and Hall algebras of rooted trees}
We begin by studying the Hall algebra associated with a planar rooted tree. We follow the article \cite{bergner20252segalsetscutsrooted} for the introduction and definitions in this section.
The main result is \cref{thmpresentationLRT}.
A tree $T$ is a set of vertices $V(T)$ and edges $E(T)$ between these vertices such that the graph is cycle-free. A rooted tree is a tree with a distinguished vertex called the root. A leaf is a vertex of degree 1. We define a branch as the unique path from a leaf to the root. There is a natural partial order on the set of vertices of a tree $T$. Two vertices $v_i$ and $v_j$ are comparable if they belong to the same branch, and in this case, $v_i < v_j$ if $v_i$ is closer to the root than $v_j$. In particular, the root is smaller than all other vertices. In what follows, our rooted tree will be labeled and planar, meaning that the edges above a vertex have a total order. In the diagram, this is represented by ordering the edges from left to right. We will often conflate a tree (or what we will call an admissible subforest) with its set of vertices, since these objects are entirely characterized by the latter. 

For the following $T$ will be a labeled planar rooted tree.
\begin{definition}
    A susbet $L$ of the vertex set of $T$ defines a lower subtree if whenever $v_j\in L$ all vertices lower than $v_j$ in the partial order are also in $L$. The complement of a lower subtree is called an upper subtree.
\end{definition}
\begin{rmk}
    A lower subtree is either empty or connected and containing the root.
\end{rmk}
\begin{definition}
    A cut on $T$ is a partition $V(T)=L\sqcup U$ such that $L$ is a lower subtree of $T$.
\end{definition}
\begin{definition}
    A layering of $n-1$ cuts of $T$ is a sequence of partially ordered subsets
    $$V(T)=L_0\supseteq L_1 \supseteq ... \supseteq L_n=\emptyset,$$
    where each $L_k$ defines a lower subtree of $T$.
\end{definition}
\begin{definition}
    A rooted forest is a disjoint union of rooted trees $F=\bigsqcup_{\alpha}T_\alpha$. 

    A subset $L=\bigsqcup_\alpha L_\alpha$ of $V(F)$ defines a lower subforest if each $L_\alpha$ is a lower subtree of $T_\alpha$. All the definitions of upper subforest and layering are analogous.

    If $V(F)=L_0\supseteq L_1 \supseteq ... \supseteq L_n=\emptyset$ is a $n-1$ layering, the subforest defined by $L_i\backslash L_j$ for $0\leq i \leq j \leq n$ are admissible subforests.
\end{definition}
\begin{definition}
    We define the following simplicial set $X(T)$ :
    \begin{enumerate}
        \item $X_0(T)=\{\emptyset\}.$
        \item $X_1(T)$ is the set of all admissible subforests of $T$.
        \item For $n\geq 2$, we define $X_n(T)=\{ (V(H)\supseteq L_1 \supseteq ... \supseteq L_n=\emptyset)~|~H\in X_1(T)\}$. 
    \end{enumerate}
    For the maps
    \begin{enumerate}
        \item $d_0(V(H)\supseteq L_1 \supseteq ... \supseteq L_n)=( L_1 \supseteq ... \supseteq L_n)$.
        \item $d_n(V(H)\supseteq L_1 \supseteq ... \supseteq L_n)=(V(H)\backslash L_{n-1}\supseteq L_1\backslash L_{n-1} \supseteq ... \supseteq L_{n-1}\backslash L_{n-1})$.
        \item For $0<i<n$, $d_i(V(H)\supseteq L_1 \supseteq ... \supseteq L_n)=(V(H)\supseteq L_1 \supseteq ...\supseteq L_{i-1} \supseteq L_{i+1} \supseteq ... \supseteq L_n)$.
        \item For $0\leq i \leq n$, the map $s_i$ just repeat the i-th cut.
    \end{enumerate}
\end{definition}
\begin{prop}\cite{bergner20252segalsetscutsrooted}
    The simplicial set $X(T)$ is 2-Segal.
\end{prop}
We define the Hall algebra associated to this 2-Segal set denoted by $\hall(T)$ as in \cref{producthallalgebra} and \cref{coproducthallalgebra}. We have the following property :
\begin{prop}\cite{bergner20252segalsetscutsrooted}
    For the Hall algebra $\hall(T)$ associated to a labelled rooted tree $T$, all multiplications are of the form $x\cdot y=\epsilon z$, where $\epsilon = 0,1$.
\end{prop}
\begin{rmk}
The coproduct of the Hall algebra of a labelled planar rooted tree is the same as that of the Connes-Kreimer algebra \cite{Connes_1998}. 
\end{rmk}
\subsection{Twisted bialgebra}
\begin{definition}
    If $T_1$ is an admissible subforest of $T$, i.e. $T_1\in X_1(T)$ and $\mathcal{B}$ a branch, we define the projection of $T_1$ on $\mathcal{B}$ by forgetting in $T_1$ all the vertices that are not in $\mathcal{B}$ (and the corresponding edges). We denote it by $\mathcal{B}(T_1)$. 
\end{definition}
\begin{rmk}
    $\mathcal{B}(T_1)$ is still an element of $X_1(T)$. Moreover, if $T_1$ has no vertices in common with $\mathcal{B}$, then, $\mathcal{B}(T_1)=\emptyset$.
\end{rmk}

 \begin{lem}\label{nulproduct}
    Let $T_1$, $T_2$, $T_3 \in X_1(T)$ and $\mathcal{B}$ a branch. Then $$\mathcal{B}(T_1)\mathcal{B}(T_2)\mathcal{B}(T_3)=0 \Longleftrightarrow \mathcal{B}(T_1)\mathcal{B}(T_2)=0 \text{ or }\mathcal{B}(T_2)\mathcal{B}(T_3)=0$$
\end{lem}
\begin{proof}
By definition, $\mathcal{B}(T_i)$ belongs to the branch $\mathcal{B}$. If the product of two term is zero it can be due to two reasons : elements are in a wrong order, i.e. the first element is closer to the root, or they are in a good order but too far, i.e. there is at least a vertex between this two elements. So, if $\mathcal{B}(T_1)\mathcal{B}(T_2)\neq 0$ and $\mathcal{B}(T_1)\mathcal{B}(T_2)\mathcal{B}(T_3)=0$ then it is because one of the two previous reasons but then the same reason implies $\mathcal{B}(T_2)\mathcal{B}(T_3)=0$.
    $\qedhere$ 
\end{proof}
Only for the case of rooted trees we have to generalize the notion of twisted bialgebra that we have introduced in \cref{deftwistedbialgebra}.
\begin{definition}
    Let $\mathcal{B}$ a branch and $A_1,A_2,B_1,B_2\in X_1(T)$. Denote by \newline  $a_1,a_2,b_1,b_2$ the projection of $A_1,A_2,B_1,B_2$ on $\mathcal{B}$. We define the test function on $\mathcal{B}$ by
\begin{equation*}
t_\mathcal{B}(A_1,A_2,B_1,B_2) = \left\{
    \begin{array}{ll}
        0 & \mbox{if } [a_2 \text{ and }b_1\neq\emptyset] \text{ or }[a_2=b_1=\emptyset \text{ and }a_1b_2=0],\\
        1 & \mbox{otherwise.}                     
    \end{array}
\right.
\end{equation*}
\end{definition}
\begin{definition}\label{deftwistedproductLRT}
    We define a new product on $\hall(T)\otimes \hall(T)$ by :
    $$(A_1\otimes A_2)\cdot(B_1\otimes B_2)=t_B(A_1,A_2,B_1,B_2)](A_1B_1\otimes A_2B_2)$$
    for $A_1,A_2,B_1,B_2\in X_1(T)$.
\end{definition}
\begin{rmk}
We can prove that the twist introduced is not of the form of a usual twist like in \cref{defusualtwist} by considering $t_\mathcal{B}(A,\emptyset,\emptyset,A)$, $t_\mathcal{B}(A,\emptyset,\emptyset,\emptyset)$ and $t_\mathcal{B}(\emptyset,\emptyset,\emptyset,A)$
\end{rmk}

\begin{prop}
    The twisted product \ref{deftwistedproductLRT} defined on $\hall(T)\otimes \hall(T)$ is associative.
\end{prop}
\begin{proof}
Let $\mathcal{B}$ a branch. We will show that the twisted multiplication 
 $$(A_1\otimes A_2)\cdot(B_1\otimes B_2)=t_B(A_1,A_2,B_1,B_2)](A_1B_1\otimes A_2B_2)$$
 is associative. Let $A_1,A_2,B_1,B_2,C_1,C_2\in X_1(T)$, the result of 
 $$((A_1\otimes A_2)(B_1\otimes B_2))(C_1\otimes C_2)$$
 can be $A_1B_1C_1\otimes A_2B_2C_2$ or $0$ (it could be the same). We just have to demonstrate that $t_\mathcal{B}$ returns zero in $((A_1\otimes A_2)(B_1\otimes B_2))(C_1\otimes C_2)$ if and only if $t_B$ returns zero in $(A_1\otimes A_2)((B_1\otimes B_2)(C_1\otimes C_2))$.
 
 
 
 \[
\begin{array}{|c|c|}
\hline
\text{Conditions on } ((\cdot)(\cdot))(\cdot)&\text{Conditions on } (\cdot)((\cdot)(\cdot))\\
\text{ such that } t_\mathcal{B} \text{ returns } 0
& 
 \text{ such that } t_\mathcal{B} \text{ returns } 0\\
\hline
\begin{array}{l}
(1)\; a_2 \text{ and } b_1 \neq \emptyset \\
(2)\; a_2 b_1 = \emptyset \text{ and } a_1 b_2 = 0 \\
(3)\; c_1 \text{ and } a_2 b_2 \neq \emptyset \\
(4)\; a_2 b_2 c_1 = \emptyset \text{ and } a_1 b_1 c_2 = 0
\end{array}
&
\begin{array}{l}
(5)\; b_2 \text{ and } c_1 \neq \emptyset \\
(6)\; b_2 c_1 = \emptyset \text{ and } b_1 c_2 = 0 \\
(7)\; b_1 c_1 \text{ and } a_2 \neq \emptyset \\
(8)\; a_2 b_1 c_1 = \emptyset \text{ and } a_1 b_2 c_2 = 0
\end{array}
\\
\hline
\end{array}
\]
 We have the following implications : 
 
 $(1)\implies (7)$ : if $b_1\neq \emptyset$ then $b_1c_1\neq \emptyset$.
 
 $(2)\implies (5)\text{ or }(8)$ : $a_1b_2=0$ implies $b_2\neq \emptyset$ if $c_1\neq \emptyset$ too, $(5)$ is verified otherwise $c_1=\emptyset$. In this case the equality $a_2b_1=\emptyset$ implies $a_2b_1c_1=\emptyset$ and $a_1b_2=0$ implies $a_1b_2c_2=0$ so $(8)$ is verified.
 
 $(3)\implies (5) \text{ or } (7)$ : if $b_2\neq \emptyset$, $(5)$ is verified, otherwise $a_2\neq \emptyset$ because $a_2=a_2b_2\neq \emptyset$. Moreover $b_1c_1\neq \emptyset$ because $c_1$ is, so $(7)$ is verified.
 
 $(4) \implies (6) \text{ if } b_1c_2=0 \text{ otherwise by the \cref{nulproduct} }a_1b_1=0$ and the product is zero in any case.

In a similar way, we can show that $(5)\implies (3)$, $(6)\implies (1)\text{ or }(4)$, $(7)\implies (1) \text{ or } (3)$, $(8) \implies (2)$ or the product is zero.
 This shows the associativity for each branch and then the associativity for the whole product by  multiplication.
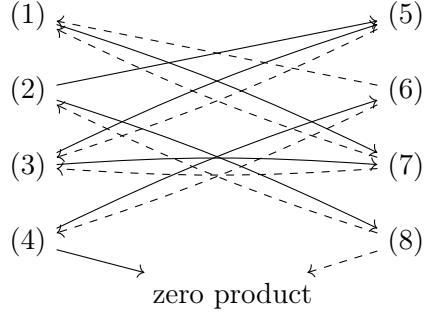
\begin{figure}[h]
    \centering
 \begin{tikzpicture}
    \node (1) {(1)};
    \node (2) [below of=1] {(2)};
    \node (3) [below of=2] {(3)};
    \node (4) [below of=3] {(4)};
    
    \node (5) [right of=1, xshift=4cm] {(5)};
    \node (6) [below of=5] {(6)};
    \node (7) [below of=6] {(7)};
    \node (8) [below of=7] {(8)};
    
    \node (P) [below right of=4, xshift=2cm] {zero product};

    \draw[->,bend left=4] (1) to (7);
    
    \draw[->] (2) to (5);
    \draw[->,bend left=4] (2) to (8);
    
    \draw[->,bend left=4] (3) to (5);
    \draw[->,bend left=4] (3) to (7);
    
    \draw[->,bend left=4] (4) to (6);
    \draw[->] (4) to (P);

    \draw[->,bend left=4,dashed] (5) to (3);
    
    \draw[->,dashed] (6) to (1);
    \draw[->,bend left=4,dashed] (6) to (4);
    
    \draw[->,bend left=4,dashed] (7) to (1);
    \draw[->,bend left=4,dashed] (7) to (3);
    
    \draw[->,bend left=4,dashed] (8) to (2);
    \draw[->,dashed] (8) to (P);

\end{tikzpicture}
\caption{Diagram of implications}
\end{figure}
 $\qedhere$
\end{proof}
\begin{thm}\label{thmbialgebraLRT}
    Let $T$ be a tree and $\hall(T)$ the associated Hall algebra and coalgebra with the canonical multiplication and comultiplication. Then $\hall(T)$ is a twisted bialgebra, i.e.,
    $$\Delta(AB)=\Delta(A)\cdot\Delta(B),$$
    where the twisted product is given by \ref{deftwistedproductLRT}.
\end{thm}
\begin{proof}
    Let $A,B\in X_1(T)$ define $C:=AB \in X_1(T)$.

    If $C=0$, then there exists a branch $\mathcal{B}$ such that $\mathcal{B}(A)\mathcal{B}(B)=0$. Now, take $A_1\otimes A_2$ and $B_1\otimes B_2$ elements respectively of $\Delta(A)$ and $\Delta(B)$, and investigate the value of the test function on $\mathcal{B}$. We denote their projection onto $\mathcal{B}$ by lowercase letters.
    If $a_2$ and $b_1\neq \emptyset$ then it is zero. If both are empty, then because $ab=0$ the test is zero too. Otherwise assume without losing of generality that $a_2$ is nonempty, because $ab=0$ we have $a_2b=a_2b_2=0$ and so $A_1B_1\otimes A_2B_2=0$.
    
    So the equality $\Delta(AB)=\Delta(A)\cdot \Delta(B)$ holds.

    If $C\neq 0$, let $C_1\otimes C_2$ be a term of $\Delta(C)$. By definition, on each branch the separation between $C_1$ and $C_2$ is in $A$ or $B$ or neither, each case excluding the others. This cut defined on each branch $\mathcal{B}$ a cut for $A$ and $B$ and at least one of them is trivial (of the form $X\otimes \emptyset$ or $\emptyset \otimes X$). It defines a cut on $A$ given by $A_1\otimes A_2$ and for $B$ by $B_1\otimes B_2$. The multiplication is 
    $$(A_1\otimes A_2)\cdot (B_1\otimes B_2)=A_1B_1\otimes A_2B_2=C_1\otimes C_2$$
    because the twist does not kill it by definition. It shows that every term of $\Delta(C)$ is included in $\Delta(A)\cdot \Delta(B)$.
    
    Conversly, consider $A_1B_1\otimes A_2B_2$ a non-zero term of $\Delta(A)\cdot \Delta(B)$. By definition of the twist, on each branch $\mathcal{B}$, we have $\mathcal{B}(A_2)=\emptyset$ or $\mathcal{B}(B_1)=\emptyset$. It defines a place to cut between $A$ and $B$ on each branch and $A_1B_1\otimes A_2B_2$ defines a cut for $C$.
    $\qedhere$
\end{proof}
\subsection{Primitive elements and presentation}
Normally we use the twisted bialgebra structure to give a presentation of the Hall algebra. Here the presentation is simple enough to be given directly. Before that we just give the following result about primitive elements. 
\begin{prop}\label{propprimitivesLRT}
    The set of vertices is a basis of the vector space $Prim(T)$.
\end{prop}
\begin{proof}
    All vertex are in $X_1(T)$ and clearly primitive. Let's take $P\in \hall(T)$ a primitive element. We have $P=\sum_ic_iH_i$ with $H_i\in X_1(T)$ all different. Then 
    $$\Delta(P)=\sum_ic_i\Delta(H_i).$$
    Because all the $H_i$ are different and by the fact that an element of $X_1(G)$ is entirely determined by its graduation, every $H_i$ must be primitive too. To conclude, in $X_1(T)$ only vertices are primitives. The linearly independence is clear.
    $\qedhere$
\end{proof}
\begin{rmk}\label{remarkgenerate}
    The algebra $\hall(T)$ is $\N$-graded by the number of vertices so the set of primitive elements generates $\hall(T)$ as an algebra (\cref{corgenerate}). We could also see it here more easily just by noticing that an element of $X_1(T)$ is given by the product of its vertices from the top the bottom.
\end{rmk}
Now we give a presentation of $\hall(T)$.
\begin{thm}\label{thmpresentationLRT}
    We have the following isomorphism :
    $$\hall(T)\simeq \Q[x_i|~i\text{ a vertex of }T]/(R),$$
    where $R$ is the set of relations given by :
    \begin{itemize}
        \item $x_ix_j=x_jx_i$ if $i$ and $j$ do not belong to the same branch.
        \item $x_ix_j=0$ if $i$ is in a common branch with $j$ and is not the vertex just above $j$.
    \end{itemize}
\end{thm}
\begin{rmk}
    Said differently, $x_i$ and $x_j$ commute if $i$ and $j$ are not in the same branch otherwise $x_ix_j$ is not zero only if $i$ is the vertex just above $j$ in the tree.
\end{rmk}
\begin{proof}
We begin this proof by showing the following result :
\begin{lem}
   Assume $j$ is above $i$ in the same branch, then for any $X=\prod_k x_k$, we have $x_iXx_j=0.$
\end{lem}
\begin{proof}
Try to pass $x_j$ through $X$ with the rule of commutation between two elements not in the same branch. If we cannot it is because there is $x_k$ with $k$ in a same branch as $j$. If $k$ is below $j$, then by definition $x_kx_j=0$, otherwise $k$ is above $j$ and now try to pass $x_k$ through the rest of $X$. Repeat this operation until the end of $X$. At the end either the product is $0$ or there is a multiplication $x_ix_k$ with $k$ above $j$ and so in the same branch as $i$. In any case $x_ix_k$ is zero. This proves that the product is zero. 
$\qedhere$
\end{proof}
   The map that sends $x_i$ to $i$ is clearly well defined and surjective via \cref{remarkgenerate}.
   Assume that $\sum_i a_i x_{i_1}...x_{i_{k_i}}=0$. By the fact that $\hall(T)$ is a semigroup on its basis, if this sum is zero it must imply that there is at least two products $x_{i_1}...x_{i_{k_i}}$ and $x_{j_1}...x_{j_{k_j}}$ that are equal. By the graduation we know that $k_i=k_j$ and that the terms in the two products are permutation of each other. So
   $$x_{i_1}...x_{i_k}=x_{j_1}...x_{j_k}.$$
   Assume that we cannot pass from one to the other by relations given above and that relations have been used to maximise the correspondence of the terms from left to right. Denote by $l$ the first integer where $x_{i_l}\neq x_{j_l}$. So the equality can be rewritten as
   $$Xx_{i_l}Y=Xx_{j_l}Y_1x_{i_l}Y_2.$$
   If $x_{i_l}$ commutes with $x_{j_l}Y_1$ it is a contradiction with the hypothesis of maximal correspondence. So there is an element $x$ in $x_{j_l}Y_1$ which does not commute with $x_{i_l}$ so either $xx_{i_l}=0$ or $x_{i_l}x=0$. The equality can be rewritten as
   $$Xx_{i_l}Y_1'xY_2'=XY_{11}xY_{12}x_{i_l}Y_2,$$
   with $Y_{11}xY_{12}=x_{j_l}Y_1$. Apply the previous lemma on $x_{i_l}Y_1'x$ and $xY_{12}x_{i_l}$ show that one of the two is zero.
    $\qedhere$
\end{proof}

\section{Undirected graphs}\label{sectionNOL}
In this section we study the Hall algebra associated to an undirected graph. The main results are 
\cref{thmbialgebraNOL} (twisted bialgebra), \cref{propprimitivesNOL} (description of primitives), \cref{thmpresentationNOL} (presentation) and \cref{thmpolyhedralproductNOL} (Hall algebra as the cohomology of a space).

An undirected graph $G$ consists of a set of vertices $V(G)$ together with a set of edges $E(G)$, where each edge is an unordered pair of vertices. Multiple edges and loops are allowed. The vertices are labelled with integers from $1$ to $n:=|V(G)|$ and edges between two vertices are also labelled. A subgraph $K$ of $G$ is given by a subset $V(K)\subseteq V(G)$ and a susbet $E(K)\subseteq E(G)$ such that the extremities of every edges in $E(K)$ belong to $V(K)$. An induced subgraph $A$ of a graph $K$ is a subgraph such that every edge between the vertices of $V(A)$ that is in $E(K)$ is in $E(A)$. Notice that an induced subgraph is only a data on vertices.
\subsection{2-Segal sets and Hall algebras of undirected graphs}
\begin{definition}
    If $A$ is an induced subgraph of $H$, we will use the notation $A\subseteq H$. In this case we will denote by $H\backslash A$ the induced subgraph of $H$ on the vertices that are not in $A$.
\end{definition}
Following \cite{bergner20172segalsetswaldhausenconstruction} we introce a 2-Segal set associated to an undirected graph.
\begin{definition}{\cite{bergner20172segalsetswaldhausenconstruction}}
    To an undirected labelled graph $G$ we associate the following simplicial set :
    \begin{itemize}
        \item $X_0(G)$ has a single element denoted by $\emptyset$.\\
        \item $X_1(G)$ is the set of all subgraphs of $G$.\\
        \item $X_n(G)=\{(S_1\subseteq...\subseteq S_n=H)~|~H\in X_1(G)\}$.
        \smallbreak
        The face map $d_i:X_n(G) \rightarrow X_{n-1}(G)$ is given by :
        \begin{itemize}
            \item if $i=0$, $d_0(S_1\subseteq...\subseteq S_n\subseteq H)=(S_2\backslash S_1\subseteq...\subseteq S_n\backslash S_1 \subseteq H\backslash S_1)$
            \item for $i\geq 1$, $d_i$ forgets the $i^{th}$ term of the filtration.
        \end{itemize}
        \smallbreak
        The map $s_i:X_n(G)\rightarrow X_{n+1}(G)$ is just given by repeating the $i^{th}$ term.
    \end{itemize}
     \end{definition}
         We will use the following graph to illustrate various properties during this section.

     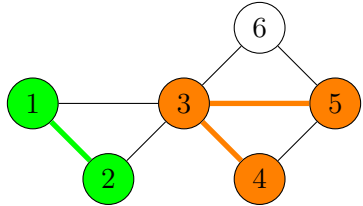
\begin{figure}[ht] \centering\begin{tikzpicture}
    \node[draw,circle, scale=1, fill=green](1) at (0,0){1};
    \node[draw,circle, scale=1, fill=green](2) at (1,-1){2};
    \node[draw,circle, scale=1, fill=orange](3) at (2,0){3};
    \node[draw,circle, scale=1, fill=orange](4) at (3,-1){4};
    \node[draw,circle, scale=1, fill=orange](5) at (4,0){5};
    \node[draw,circle, scale=1](6) at (3,1){6};
    \draw[color=green,line width=2pt] (1)--(2);
    \draw (1)--(3);
    \draw (2)--(3);
    \draw[color=orange,line width=2pt] (3)--(5);
    \draw[color=orange,line width=2pt] (3)--(4);
    \draw (4)--(5);
    \draw (3)--(6);
    \draw (5)--(6);
 \end{tikzpicture} 
 \caption{A graph $G$ with two subgraphs : $H$ in green and $K$ in orange.} \label{exmulti}
\end{figure}
     \begin{ex}
         For the graph $G$ of \cref{exmulti} :
         \begin{itemize}
             \item $(H\subseteq G)\in X_2(G)$
             \item $(\emptyset \subseteq H\subseteq H)\in X_3(G)$
             \item $(K\subseteq G)\notin X_2(G)$ because $K$ is not induced in $G$
             \item $(3\subseteq K)\in X_2(G)$ (where $3$ denotes the graph with only the vertex $3$)
         \end{itemize}
         
     \end{ex}
    \begin{prop}{\cite{bergner20172segalsetswaldhausenconstruction}}
        This simplicial set is 2-Segal.
    \end{prop}
    \begin{definition}
        Let $A,B\in X_1(G)$ two vertices disjoints subgraphs of $G$ we define $E(A,B)$ (or sometimes $E_G(A,B)$ if it is not clear) to be the set of edges between $A$ and $B$ in $E(G)$. If $e \subseteq E(A,B)$, we define $Ae B$ to be the subgraph of $G$ with vertices $v(A)\sqcup v(B)$ and edges $E(A)\sqcup E(B) \sqcup e$.
    \end{definition}
    By the 2-Segal condition, we obtain a natural associative product on $\hall(G)$ the Hall algebra associated to $X(G)$(see \cref{producthallalgebra}).
    \begin{prop}\label{propdefmultiNOL}
        In our context the product (\cref{producthallalgebra}) can be reinterpreted as follow. Let $A,B\in X_1(G)$, then 
       \begin{equation*} A B = \left\{
    \begin{array}{ll}
        \sum_{e \subseteq E(A,B)} Ae B ~~\mbox{ if }A\cap B=\emptyset,\\
        0 ~~\mbox{ otherwise. }
    \end{array}
\right.\end{equation*}
    \end{prop}
    \begin{proof}
    We are searching for elements $(A\subseteq H)$ of $X_2(G)$ such that $H\backslash A=B$. it is easy to see that $H$ is of the form $AeB$ and that such $H$ give a valid element of $X_2(G)$. To conclude notice that all coefficients are $1$.
    $\qedhere$
    \end{proof}
    \begin{ex}
    Let's compute the product $H K$ as represented in \cref{exmulti}.

\begin{align*}
    H K =& \sum_{e\subseteq \{1-3,2-3\}}HeK \\ =&\vcenter{\hbox{\begin{tikzpicture}
    \node[draw,circle, scale=.5, fill=green](1) at (0,0){1};
    \node[draw,circle, scale=.5, fill=green](2) at (0.5,-0.5){2};
    \node[draw,circle, scale=.5, fill=orange](3) at (1,0){3};
    \node[draw,circle, scale=.5, fill=orange](4) at (1.5,-0.5){4};
    \node[draw,circle, scale=.5, fill=orange](5) at (2,0){5};
    \draw (1)--(2);
    \draw (3)--(5);
    \draw (3)--(4);
    \end{tikzpicture}}}
    +
    \vcenter{\hbox{\begin{tikzpicture}
    \node[draw,circle, scale=.5, fill=green](1) at (0,0){1};
    \node[draw,circle, scale=.5, fill=green](2) at (0.5,-0.5){2};
    \node[draw,circle, scale=.5, fill=orange](3) at (1,0){3};
    \node[draw,circle, scale=.5, fill=orange](4) at (1.5,-0.5){4};
    \node[draw,circle, scale=.5, fill=orange](5) at (2,0){5};
    \draw (1)--(2);
    \draw (1)--(3);
    \draw (3)--(5);
    \draw (3)--(4);
    \end{tikzpicture}}}
    +
    \vcenter{\hbox{\begin{tikzpicture}
    \node[draw,circle, scale=.5, fill=green](1) at (0,0){1};
    \node[draw,circle, scale=.5, fill=green](2) at (0.5,-0.5){2};
    \node[draw,circle, scale=.5, fill=orange](3) at (1,0){3};
    \node[draw,circle, scale=.5, fill=orange](4) at (1.5,-0.5){4};
    \node[draw,circle, scale=.5, fill=orange](5) at (2,0){5};
    \draw (1)--(2);
    \draw (2)--(3);
    \draw (3)--(5);
    \draw (3)--(4);
    \end{tikzpicture}}}
+
\vcenter{\hbox{\begin{tikzpicture}
    \node[draw,circle, scale=.5, fill=green](1) at (0,0){1};
    \node[draw,circle, scale=.5, fill=green](2) at (0.5,-0.5){2};
    \node[draw,circle, scale=.5, fill=orange](3) at (1,0){3};
    \node[draw,circle, scale=.5, fill=orange](4) at (1.5,-0.5){4};
    \node[draw,circle, scale=.5, fill=orange](5) at (2,0){5};
    \draw (1)--(2);
    \draw (1)--(3);
    \draw (2)--(3);
    \draw (3)--(5);
    \draw (3)--(4);
    \end{tikzpicture}}}.
\end{align*}

\end{ex}
The 2-Segal condition also gives us a natural coproduct (\cref{coproducthallalgebra}).

\begin{prop}
In our context the coproduct can be reinterpreted as follow. 
    For $A\in X_1(G)$, we have
    $$\Delta(A)=\sum_{K\subseteq A}K\otimes A\backslash K.$$
\end{prop}
\begin{proof}
This proof is dual to the one for the multiplication.
$\qedhere$
\end{proof}
\begin{rmk}
    The product and the coproduct are respectively commutative and cocommutative and all the coefficients in a product or a coproduct concerning basis elements are $1$. The commutativity is clear since every $(A\subseteq H)\in X_2(G)$ induces $(H\backslash A\subseteq H)$. The commutativity of the Hall algebra of a finitary proto-exact category $\mathcal{C}$ is often due to the existence of an exact equivalence $\mathcal{D}:\mathcal{C}^{op}\rightarrow \mathcal{C}$ such that $M\simeq D(M)$. We will see that in our case we are no exception to the rule.
\end{rmk}
\begin{ex}
Consider the subgraph $K$ of $G$ in \cref{exmulti}.
    \begin{align*}
        \Delta(K)= &
        ~\emptyset \otimes 
        \vcenter{\hbox{\begin{tikzpicture}
    \node[draw,circle, scale=.5, fill=orange](3) at (1,0){3};
    \node[draw,circle, scale=.5, fill=orange](4) at (1.5,-0.5){4};
    \node[draw,circle, scale=.5, fill=orange](5) at (2,0){5};
    \draw (3)--(5);
    \draw (3)--(4);
    \end{tikzpicture}}}
    +
    \vcenter{\hbox{\begin{tikzpicture}
    \node[draw,circle, scale=.5, fill=orange](3) at (1,0){3};
    \end{tikzpicture}}}
\otimes 
\vcenter{\hbox{\begin{tikzpicture}
    \node[draw,circle, scale=.5, fill=orange](4) at (1.5,-0.5){4};
    \node[draw,circle, scale=.5, fill=orange](5) at (2,0){5};
    \end{tikzpicture}}}
    +
   \vcenter{\hbox{ \begin{tikzpicture}
    \node[draw,circle, scale=.5, fill=orange](4) at (1.5,-0.5){4};
    \end{tikzpicture}}}
    \otimes
    \vcenter{\hbox{\begin{tikzpicture}
    \node[draw,circle, scale=.5, fill=orange](3) at (1,0){3};
    \node[draw,circle, scale=.5, fill=orange](5) at (2,0){5};
    \draw (3)--(5);
    \end{tikzpicture}}}
    +
    \vcenter{\hbox{\begin{tikzpicture}
    \node[draw,circle, scale=.5, fill=orange](5) at (2,0){5};
    \end{tikzpicture}}}
    \otimes
    \vcenter{\hbox{\begin{tikzpicture}
    \node[draw,circle, scale=.5, fill=orange](3) at (1,0){3};
    \node[draw,circle, scale=.5, fill=orange](4) at (1.5,-0.5){4};
    \draw (3)--(4);
    \end{tikzpicture}}}
    \\ &
    + 
    \vcenter{\hbox{\begin{tikzpicture}
    \node[draw,circle, scale=.5, fill=orange](3) at (1,0){3};
    \node[draw,circle, scale=.5, fill=orange](4) at (1.5,-0.5){4};
    \draw (3)--(4);
    \end{tikzpicture}}}
    \otimes
    \vcenter{\hbox{\begin{tikzpicture}
    \node[draw,circle, scale=.5, fill=orange](5) at (2,0){5};
    \end{tikzpicture}}}
    +
    \vcenter{\hbox{\begin{tikzpicture}
    \node[draw,circle, scale=.5, fill=orange](3) at (1,0){3};
    \node[draw,circle, scale=.5, fill=orange](5) at (2,0){5};
    \draw (3)--(5);
    \end{tikzpicture}}}
    \otimes
    \vcenter{\hbox{\begin{tikzpicture}
    \node[draw,circle, scale=.5, fill=orange](4) at (1.5,-0.5){4};
    \end{tikzpicture}}}
    +
    \vcenter{\hbox{\begin{tikzpicture}
    \node[draw,circle, scale=.5, fill=orange](4) at (1.5,-0.5){4};
    \node[draw,circle, scale=.5, fill=orange](5) at (2,0){5};
    \end{tikzpicture}}}
    \otimes
    \vcenter{\hbox{\begin{tikzpicture}
    \node[draw,circle, scale=.5, fill=orange](3) at (1,0){3};
    \end{tikzpicture}}}
    +
    \vcenter{\hbox{\begin{tikzpicture}
    \node[draw,circle, scale=.5, fill=orange](3) at (1,0){3};
    \node[draw,circle, scale=.5, fill=orange](4) at (1.5,-0.5){4};
    \node[draw,circle, scale=.5, fill=orange](5) at (2,0){5};
    \draw (3)--(5);
    \draw (3)--(4);
    \end{tikzpicture}}}
    \otimes \emptyset.
    \end{align*}
\end{ex}
\begin{rmk}
\begin{enumerate}
   \item The coproduct depends only the graph $K$ and not on $G$.
   \item Let $V$ be the set of vertices of $G$. Then $\hall(G)$ is $\N V$-graded, indeed every $A\in X_1(G)$ belongs to vector space graded by $\sum_{v\in V(A)}v$. Moroever the product and coproduct respect the grading.
\end{enumerate}
\end{rmk}
\subsection{Proto-exact category of an undirected graph}
We now examine whether the product and coproduct endow $\hall(G)$ with the structure of a bialgebra. It turns out that this is not the case in a direct manner. In analogy with Green’s theorem in the setting of hereditary abelian categories, we introduce a twist in the multiplication on $\hall(G)\otimes \hall(G)$.
A direct proof that $\hall(G)$ admits the structure of a twisted bialgebra can be obtained by comparing both sides of the identity $\Delta(AB)=\Delta(A)\cdot \Delta(B)$ for $A,B\in X_1(G)$. This argument closely parallels those given in \cref{bigebrecoproduitnaif} and \cref{thmbialgebraOL}, and is therefore left to the reader.
Instead, we present an alternative proof based on the same underlying ideas as the proof of Green’s theorem given by Dyckerhoff in \cite{Dyckerhoff_2018}. We refer to this work for a definition of the abstract Hall algebra.

To provide a similar proof, we begin by introducing a proto-exact category of a graph that yields the same Hall algebra as before.
\begin{definition}
    Let $G$ be a graph and let $\mathcal{C}_G$ be the category with objects subgraphs of $G$ and for $H,K\in \mathcal{C}_G$, Hom$(H,K):=\{L\text{ subgraph of }G \text{ induced in }H \text{ and } K\}$. The composition is given by intersection and the identity of a subgraph $H$ is $H$ itself.
\end{definition}
\begin{rmk}
   One may regard $\mathcal{C}_G$ as generated by the associated pointed stable double category (see \cite{bergner20172segalsetswaldhausenconstruction}), where horizontal and vertical morphisms are merged into a single class.
\end{rmk}

\begin{thm}
    The category $\mathcal{C}_G$ is a finitary and cofinitary proto-exact category (see \cite{DyckerhoffKapranov2019} definition 2.4.2) where the class of admissible monomorphisms $\mathfrak{M}$ is given by all the arrows of the form $H\overset{H}{\rightarrow}K$. Similarly the class of admissible epimorphisms $\mathfrak{E}$ is given by all the arrows of the form $H\overset{K}{\rightarrow}K$.
\end{thm}
\begin{proof}
Finitary and cofinitary aspect are obvious because all is finite.
    The category $\mathcal{C}_G$ is pointed by $\emptyset$ and every morphism $\emptyset\overset{\emptyset}{\rightarrow}K$ belongs to $\mathfrak{M}$ similarly $K\overset{\emptyset}{\rightarrow}\emptyset$ belongs to $\mathfrak{E}$. 
    
    Obviously the classes $\mathfrak{M}$ and $\mathfrak{E}$ are closed under composition and contain all the isomorphisms (only the identity in this case).

    Now take a diagram 
     \begin{equation}
    \xymatrix{ & B \ar[d]_D \\
	C \ar[r]_C & D}
\end{equation}
   and complete it by forming the diagram 
    \begin{equation}
    \xymatrix{A \ar[r]_{A}\ar[d]_{C} & B \ar[d] \\
	C \ar[r] & D}
\end{equation}
where $A:=B\backslash(D\backslash C)$. It is possible because we know that $D$ is an induced subgraph of $C$ which is itself an induced subgraph of $B$. Then vertical morphisms of this diagram are admissible epis and horizontals are admissible monos.

Now we check that it is Cartesian. Take $K$ a subgraph of $G$ with morphisms $$F_1:K\rightarrow B \text{ and } F_2:K\rightarrow C$$ such that the diagram is commutative i.e. $F_1\cap D=F_2\cap C=F_2$ (cause $F_2$ is induced in $C$). It follows from the equality that $F_1\cap C=F_2$ and $F_1\cap D\backslash C=\emptyset$. Consider the arrow $F_1:K\rightarrow A$, it is straighforward to verify that the diagram is commutative, so $A$ is a pullback. 

The dual property works the same way by completing a diagram
 \begin{equation}
    \xymatrix{A \ar[r]_{A}\ar[d]_{C} & B  \\
	C  & }
\end{equation}
by 
\begin{equation}\label{diagramfull}
    \xymatrix{A \ar[r]\ar[d] & B \ar[d]_D \\
	C \ar[r]_C & D}
\end{equation}
with $D=B\backslash(A\backslash C)$.

Now it remains to check that a diagram is Cartesian if and only if it is coCartesian. As we saw previously, a diagram as \ref{diagramfull} is Cartesian if and only if $A=B\backslash(D\backslash C)$ (with the notation above), and coCartesian if and only if  $D=B\backslash(A\backslash C)$. Now the equalies $A=B\backslash((B\backslash (A\backslash C))\backslash C)$ and $D=B\backslash((B\backslash (D\backslash C))\backslash C)$ show that a diagram is Cartesian if and only if it is coCartesian.
    $\qedhere$
\end{proof}
\begin{prop}
    The Hall algebra $\hall(G)$ is the same as the Hall algebra associated to the proto-exact category $\mathcal{C}_G$.
\end{prop}
\begin{proof}
Just remark that a flag $(A_1\hookrightarrow A_2 \hookrightarrow ... \hookrightarrow A_n)$ from the Waldhausen construction in $\mathcal{C}_G$ is the same as $(A_1\subseteq A_2 \subseteq... \subseteq A_n) \in X_n(G)$ and maps $d_i$ and $s_i$ are the same.
$\qedhere$
\end{proof}
\begin{prop}\label{exactequivalence}
Let $\Phi$ be the functor from $\mathcal{C}_G^{op}$ that sends an object $K$ to itself and a morphism $H\overset{L}{\rightarrow}K\in Hom_{C_G^{op}}(K,H)$ to $K\overset{L}{\rightarrow}H\in Hom_{C_G}(K,H)$.
 Then $\Phi$ is an exact equivalence with $\Phi(M)=M$, so the Hall algebra of $C_G$ is commutative.
 \end{prop}
 \begin{proof}
 It is straightforward to verify that $\Phi$ is a strict involution.
 $\qedhere$
 \end{proof}
 \subsection{Twisted bialgebra}
\begin{definition}
    Let $A,B,C\in X_1(G)$, we will say that
    $$A\rightarrow B \rightarrow C$$
    is an exact sequence if it is of the form
    $$A\overset{A}{\rightarrow}B\overset{C}{\rightarrow}C,$$
    with $B=AeC$ for a set $e\subseteq E(A,C)$.
    \end{definition}
\begin{definition}\label{deftwistNOL}
Let's define $\tilde{\phi}$ to be 
    \begin{equation*}
\begin{array}{cccc}
    &\tilde{\phi} : V\times V& \rightarrow &\Z\sqcup \{-\infty\}   \\
     &(v_1,v_2) &\mapsto& \left\{
    \begin{array}{ll}
        -\infty ~~\mbox{ if }v_1=v_2\\
        |E(v_1,v_2)| ~~\mbox{ otherwise. }
    \end{array}\right.
\end{array}
\end{equation*}
It defines a bilinear application of monoids $\phi : \N V \times \N V \rightarrow \Z\sqcup \{-\infty\}$.
\end{definition}
\begin{ex}
    For $H$ and $K$ in \cref{exmulti}, $\phi(H,K)=2$ but $\phi(v_1,H)=-\infty$ where $v_1$ is the graph with only the vertex $1$.
\end{ex}
\begin{rmk}
    If $A,B\in X_1(G)$, the result of $\phi(A,B)$ is $|E(A,B)|$ if $A\cap B=\emptyset$ and $-\infty$ otherwise.
\end{rmk}
\begin{thm}\label{thmbialgebraNOL}
    Let $G$ be an undirected labelled graph, then $(\hall(G),\cdot,\Delta,\emptyset,\epsilon)$ is a $(\Q,2,\phi,\phi)$-bialgebra.
\end{thm}

\begin{proof}
We start by computing $\Delta(A B)$ for $A,B\in X_1(G)$.
As in \cite{Dyckerhoff_2018} (section 2.4.1), it is the span 
\begin{equation*}
    \xymatrix{ & \ar[dl]_l \cross \ar[dr]^r &\\ X_1(G) \times
		X_1(G) & & X_1(G) \times X_1(G)}
\end{equation*}
as we are working in the abstract Hall algebra,
where $\cross$ denotes the set of diagrams of the form
\begin{equation} \label{diagcross}
    \xymatrix{ & A \ar[d] & \\
			C_1 \ar[r] & C \ar[d]\ar[r] & C_2\\
		& B & }
\end{equation}
called crosses, where $A,B,C,C_1,C_2$ belongs to $X_1(G)$ and with $A\rightarrow C \rightarrow B$, $C_1\rightarrow C \rightarrow C_2$ are exact sequences in $\mathcal{C}_G$. The application $l$ returns $(A,B)$ and $r$ returns $(C_1,C_2)$.
\bigbreak The other side $\Delta(A)\Delta(B)$ (considered for the moment with the usual multiplication $(A_1\otimes A_2)(B_1\otimes B_2)=A_1B_1\otimes A_2B_2$) is given by the span
\begin{equation}\label{spanframe}
    \xymatrix{ & \ar[dl]_m \cadre \ar[dr]^n &\\ X_1(G) \times
		X_1(G) & & X_1(G) \times X_1(G)}
\end{equation}
where $\cadre$ is the set of diagrams of the form 
\begin{equation}\label{diagframe}
    \xymatrix{A_1 \ar[r]\ar[d] & A \ar[r] & A_2\ar[d] \\
		C_1 \ar[d] &  & C_2\ar[d] \\
	B_1 \ar[r] & B \ar[r] & C_2}
\end{equation}
called frames, where all rows and columns are exact sequences. The application $m$ returns $(A,B)$ and $n$ returns $(C_1,C_2)$.
If $\hall(G)$ were a genuine bialgebra the sets $\cross$ and $\cadre$ should be in a bijection. It will follow that it is not the case. To compare them like in \cite{Dyckerhoff_2018} we introduce the set $\carre$ of diagrams of the form
\begin{equation} \label{squarefig}
    \xymatrix{A_1 \ar[r]\ar[d] & A \ar[r] \ar[d] & A_2\ar[d] \\
		C_1 \ar[r]\ar[d] & C \ar[d]\ar[r] & C_2\ar[d]\\
	B_1 \ar[r] & B \ar[r] & B_2 }
    \end{equation}
    called square, where all rows and columns are exact sequences.
\begin{lem}
    The sets $\cross$ and $\carre$ are in bijection.
\end{lem}
\begin{proof}
    Consider a cross 
    \begin{equation*} 
    \xymatrix{ & A \ar[d] & \\
			C_1 \ar[r] & C \ar[d]\ar[r] & C_2\\
		& B & }
\end{equation*} 
and try to complete it in a square. The subgraphs $A$ and $C_1$ are induced subgraphs of $C$ and we want in the top left corner an induced subgraph common to $A$ and $C_1$. Assume that a vertex $i$ in common with $A$ and $C_1$ is not in the subgraph that we choose and that we succeed to complete it in a square 
\begin{equation*}
    \xymatrix{A_1 \ar[r]\ar[d] & A \ar[r] \ar[d] & A_2\ar[d] \\
		C_1 \ar[r]\ar[d] & C \ar[d]\ar[r] & C_2\ar[d]\\
	B_1 \ar[r] & B\ar[r] & B_2 }.
    \end{equation*}
    The vertex $i$ will be in the graph $A_2$ so in $C_2$ too but the vertex $i$ will be in $C_1$ and $C_2$ that absurd. So if we can complete a cross to a square the graph on the top left corner must be $A\cap C_1$, i.e. the induced subgraph on the vertices in common with $A$ and $C_1$. This result is the same for all the remaining graphs. From the cross we obtain the following square :
    \begin{equation*}
    \xymatrix{A\cap C_1 \ar[r]\ar[d] & A \ar[r] \ar[d] & A\cap C_2\ar[d] \\
		C_1 \ar[r]\ar[d] & C \ar[d]\ar[r] & C_2\ar[d]\\
	B\cap C_1 \ar[r] & B\ar[r] & B\cap C_2 }.
    \end{equation*}
    From a square, we can just forget all the corners to obtain a cross. This two maps are easily seen to be inverse and conclude the proof.
    $\qedhere$
\end{proof}
Now we have to investigate the connection between frame and square. Given a square, we can just forget the graph in the middle to obtain a frame but not all the frames can be obtain like this and various squares can give the same frame just because they differ in the middle graph. 
\begin{lem}
    Fix a frame $F$ such as \cref{diagframe}. If $A_1\cap B_2=A_2\cap B_1=\emptyset$ there is $2^{|E(A_1,B_2)|+|E(A_2,B_1)|}$ square that give the frame $F$, otherwise there is no square that give the frame $F$.
\end{lem}
\begin{proof}
Let be a square like in \cref{squarefig}, the middle term $C$ has vertices composed by those of $A_1,A_2,B_1$ and $B_2$. So if $A_1\cap B_2=A_2\cap B_1=\emptyset$ the frame $F$ can not be fill into a square. Otherwise $C$ is the graph composed by $A_1,A_2,B_1$ and $B_2$ and a set of edges between this four graphs. Edges between $A_1$ and $A_2$ are already determine by $A$ the same for $B_1$, $B_2$ by $B$, $A_1$, $B_1$ by $C_1$ and $A_2$, $B_2$ by $C_2$. So it remains to choose edges between $A_1$, $B_2$ and $A_2$, $B_1$ and all such choice gives a valid choice for $C$, so there is $2^{|E(A_1,B_2)|+|E(A_2,B_1)|}$ squares that give the frame $F$ after forgetting the middle term.
    $\qedhere$
    \end{proof}
    The previous lemma shows precisely that the twisted multiplication given by $\phi$ (\cref{deftwistNOL}) removes exactly those elements of $\Delta(A)\Delta(B)$ that do not appear in $\Delta(AB)$ and multiplies the remaining elements by the correct coefficient so that $\Delta(A)\cdot \Delta(B)=\Delta(AB)$.
    $\qedhere$
\end{proof}
\subsection{Primitive elements and presentation}
\begin{definition}
    Let $H$ be a subgraph of $G$ and $e\subseteq E(H)$ be a set of edges of $H$. We denote by $H_e$ the subgraph of $G$ with vertices $V(H)$ and with $e$ for edges.
\end{definition}
\begin{definition}
    Let $H$ be a subgraph of $G$. To $H$ we associate the element $S_H$ defined as 
    $$S_H=\sum_{e\subseteq E(H)}(-1)^{|e|}H_e.$$
\end{definition}
\begin{ex}
    For the graph $K$ of \cref{exmulti}, 
    \begin{center}
        
    $S_K:=\vcenter{\hbox{\begin{tikzpicture}
    \node[draw,circle, scale=.5, fill=orange](3) at (1,0){3};
    \node[draw,circle, scale=.5, fill=orange](4) at (1.5,-0.5){4};
    \node[draw,circle, scale=.5, fill=orange](5) at (2,0){5};
    \draw (3)--(5);
    \draw (3)--(4);
    \end{tikzpicture}}}
    - 
    \vcenter{\hbox{\begin{tikzpicture}
    \node[draw,circle, scale=.5, fill=orange](3) at (1,0){3};
    \node[draw,circle, scale=.5, fill=orange](4) at (1.5,-0.5){4};
    \node[draw,circle, scale=.5, fill=orange](5) at (2,0){5};
   
    \draw (3)--(4);
    \end{tikzpicture}}}
    -
    \vcenter{\hbox{\begin{tikzpicture}
    \node[draw,circle, scale=.5, fill=orange](3) at (1,0){3};
    \node[draw,circle, scale=.5, fill=orange](4) at (1.5,-0.5){4};
    \node[draw,circle, scale=.5, fill=orange](5) at (2,0){5};
    \draw (3)--(5);
   
    \end{tikzpicture}}}
    +
    \vcenter{\hbox{\begin{tikzpicture}
    \node[draw,circle, scale=.5, fill=orange](3) at (1,0){3};
    \node[draw,circle, scale=.5, fill=orange](4) at (1.5,-0.5){4};
    \node[draw,circle, scale=.5, fill=orange](5) at (2,0){5};

    \end{tikzpicture}}}$
    
    \end{center}
    
\end{ex}
\begin{rmk}\label{remgraduationSh}
    Each term in the expression of $S_H$ has the same grading as $H$. Each term in $S_H$ is a subgraph with an equal or lower number of edges than $H$.
\end{rmk}
\begin{prop}
    The set $\{S_H,H\in X_1^G\}$ is a basis of the vector space $\hall(G)$.
\end{prop}
\begin{proof}
   We choose an ordering on $X_1(G)$ such that the elements are sorted in ascending order, first by the number of vertices and then by the number of edges. From \cref{remgraduationSh}, we can see that the transition matrix of the map $S_\cdot$ is an upper triangular matrix with $\pm 1$'s on the diagonal. 
$\qedhere$
\end{proof}
\begin{prop}
    Let $H$ be a nonempty connected subgraph of $G$. Then $S_H$ is a primitive element i.e. 
    $$\Delta(S_H)=\emptyset\otimes S_H+S_H \otimes \emptyset.$$
\end{prop}
\begin{proof}
$$\Delta(S_H)=\sum_{E\subseteq E(H)}(-1)^{|E|}\Delta(H_E).$$
Let $E\subseteq E(H)$ and  $M\otimes N$ a term of $\Delta(H_E)$. Let $E'\subseteq E(H)$, then
$$M\otimes N \text{ appears in }\Delta(H_{E'})\Longleftrightarrow E(M\sqcup N) \subseteq E'\subseteq E(M\sqcup N)\sqcup E_H(M,N),$$ where $E_H(M,N)=E(M,N)\cap E(H)$. Moreover if $M\otimes N$ appears in $\Delta(H_{E'})$, it appears exactly once. Assume that $d:=|E_H(M,N)|\geq 1$. After taking the projection on the factor $M\otimes N $ we calculate the sum 
\begin{align*}
   \sum_{E(M\sqcup N) \subseteq E'\subseteq E(M\sqcup N)\sqcup E_H(M,N)}&(-1)^{|E'|}M\otimes N\\
    &= \sum_{P\subseteq E_H(M,N)}(-1)^{|E(M\sqcup N)|+|P|}M\otimes N
    \\ &= (-1)^{|E(M\sqcup N)|}\left(\sum_{k=0}^d \binom{d}{k}(-1)^k\right)M\otimes N
    \\ &= 0.
\end{align*}
Now if $E(M,N)\cap E(H)=\emptyset$, it means that there is no edges in $H$ between $M$ and $N$, it is not possible because $H$ is supposed to be connected unless $(M,N)$ is $(H_E,\emptyset)$ or $(\emptyset,H_E)$. So
\begin{align*}
    \Delta(S_H)&=\sum_{E\subseteq E(H)}(-1)^{|E|}\Delta(H_E) \\
    &= \sum_{E\subseteq E(H)}(-1)^{|E|} (\emptyset\otimes H_E + H_E \otimes \emptyset) \\
    &= \emptyset \otimes S_H + S_H \otimes \emptyset.
\end{align*}
    $\qedhere$ 
\end{proof}
\begin{definition}
    Denote by $Con(G)$ the set of all nonempty connected subgraphs of $G$.
\end{definition}

\begin{prop}\label{propprimitivesNOL}
     The set $\{S_H,H\in Con(G)\}$ is a basis of the vector space $Prim(G)$.
\end{prop}
\begin{proof}
Let $P$ be a primitive element and $P=\sum_i\lambda_iS_{H_i}$ its decomposition in the basis $\{S_{H}~|~H\in X_1(G)\}$, so with $\lambda_i\neq 0$, for all i.
\newline First, we check that $H_i\neq \emptyset$ for all $i$, so that $\langle P,\emptyset\rangle=0$. Indeed 
\begin{align*}
    \langle P,\emptyset\rangle= \langle P,\emptyset\cdot \emptyset \rangle = \langle \Delta(P),\emptyset \otimes \emptyset\rangle = \langle \emptyset \otimes P + P \otimes \emptyset, \emptyset \otimes \emptyset \rangle =0.
\end{align*}
So $P=\sum_i\lambda_iS_{H_i}$ with $H_i\neq \emptyset$ for all $i$. More precisely, $P=\sum_j\lambda_jS_{H_j}+\sum_k\mu_kS_{K_k}$ with $H_j$ connected and $K_k$ not connected. As $P$ is primitive and $\sum_j\lambda_jS_{H_j}$ too, it implies that $\sum_k\mu_kS_{K_k}$ is primitive too. Among all the subraghs $K_k$, choose one with the maximal number of edges, we denote it by $\tilde{K}$. As $\tilde{K}$ is not connected and nonempty, we can find $I$ and $\bar{I}$ nonempty induced subgraphs of $\tilde{K}$ such that $\tilde{K}=I\sqcup\bar{I}$. We claim that $$\langle \sum_k\mu_kS_{K_k},I \cdot \bar{I}\rangle\neq 0.$$ Indeed on the left side of the scalar product, it is a sum with terms that are all subgraphs with a lower or equal number of edges than $\tilde{K}$. If another $K_a$ has the same number of edges it must has different vertices, otherwise $K_a=\tilde{K}$. So, on the left side there is only one term having the same amount of edges and the same vertices than $\tilde{K}$ that is $\tilde{K}$ itself. On the right side, all the terms in the product of $I\cdot \bar{I}$ have the same vertices than $\tilde{K}$ and strictly more edges than $\tilde{K}$ except one, that is $\tilde{K}=I\sqcup \bar{I}$. 

So $$\langle \sum_k\mu_kS_{K_k},I\cdot \bar{I}\rangle = \tilde{\mu}\times(-1)^{|E(\tilde{K})|}$$ and $$\langle \sum_k\mu_kS_{K_k},I\cdot \bar{I}\rangle=\langle \Delta(\sum_k\mu_kS_{K_k}),I\otimes \bar{I\rangle}=0$$ as $\sum_k\mu_kS_{K_k}$ is primitive. It is a contradiction and $\mu_k=0$ for all $k$.
    $\qedhere$
\end{proof}
\begin{rmk}
    $\hall(G)$ is $\N$-graded by the number of vertices (after forgetting the labels).
\end{rmk}

\begin{thm}\label{thmpresentationNOL}
    We have the following isomorphism
    $$\hall(G)\simeq \Q[x_H|H\in Con(G)]/\!\raisebox{-.65ex}{\ensuremath{(x_Hx_K|H\cap K\neq \emptyset)}}$$
\end{thm}
\begin{proof}
Consider $f$ that maps $x_H$ to $S_H$. It is clear that the algebra morphism $f$ is well defined since $$S_H S_K=0 \text{ if } H\cap K\neq \emptyset.$$using \cref{propdefmultiNOL} and \cref{remgraduationSh}. Due to the previous remark we can apply \cref{corgenerate} to show that the map is surjective.  For injectivity use \cref{proprelationprimitifs} and use the commutativity.
    $\qedhere$
\end{proof}

\subsection{A parallel with classical Hall algebras}
By the term classical Hall algebra we mean the Hall algebra of the category of nilpotent representations of a quiver $\overset{\rightarrow}{Q}$ over a finite field $\F_q$, denoted $\hall_{\overset{\rightarrow}{Q}}$ ; see \cite{schiffmann2009lectureshallalgebras} for more details. 
In this section we try to investigate the similarities between the twist on graphs and the other one introduced by Green in \cite{greenhereditaryalgebra1995} for $\hall_{\overset{\rightarrow}{Q}}$. Let's just remember the classical twist : let $A,B,A',B'$ be objects in $Rep_{\F_q}(\overset{\rightarrow}{Q})$ then
\begin{equation*}
        (\mathbf{1}_A \otimes \mathbf{1}_B) (\mathbf{1}_{A'} \otimes \mathbf{1}_{B'}) :=  \frac{|\text{Ext}^1(A',B)|}{|\text{Hom}(A',B)|}(\mathbf{1}_A \mathbf{1}_{A'})
			\otimes (\mathbf{1}_B \mathbf{1}_{B'}).
    \end{equation*}
This twist is well-defined in an abelian category, whereas the proto-exact category associated with a graph is not abelian. We therefore need to adapt the standard terminology of homological algebra slightly in our case.
Recall that for $A,B,C\in X_1(G)$, 
    $$A\rightarrow B \rightarrow C$$
    is an exact sequence if it is of the form
    $$A\overset{A}{\rightarrow}B\overset{C}{\rightarrow}C$$
    with $B=AeC$ for a set $e\subseteq E(A,C)$.
\begin{definition}
    With the same anology we define long exact sequences to be
    $$A_0\overset{B_0}{\rightarrow} A_1 \overset{B_1}{\rightarrow} ... \overset{B_{n-2}}{\rightarrow} A_{n-1} \overset{B_{n-1}}{\rightarrow} A_n $$
    with $B_0=A_0$, $B_{n-1}=A_n$ and $B_{i+1}=A_{i+1}\backslash B_i$ for $0\leq i \leq n-2$.

    We define $Ext^n_{\mathcal{C}_G}(A,B)$ to be the usual Yoneda equivalence relation between exact sequences of length $n$.
    
    We say that $\mathcal{C}_G$ is of global dimension $n$ if $|Ext^k(A,B)|=1$ for all $A,B\in X_1(G)$ and $k>n$ and there exist $P$ and $Q$ such that $|Ext^n(P,Q)|>1$.
\end{definition}
\begin{prop}
    Let $G$ be a graph. If $G$ contains at least two vertices with an edge between them then $\mathcal{C}_G$ is hereditary, i.e., $\mathcal{C}_G$ is of global dimension one. Otherwise, $G$ is empty or a disjoint union of vertices with loops and $\mathcal{C}_G$ is of global dimension $0$.
\end{prop}
\begin{proof}
    If $G$ contains two vertices linked by an edge, denote by $1$ and $2$ the subgraphs with only one vertex and by $12$ and $\underline{12}$ respectively the subgraphs with two vertices and no edge and with two vertices and an edge between them. Consider the two exact sequences 
    $$1\overset{1}{\rightarrow}12\overset{2}{\rightarrow}2 \text{ and }1\overset{1}{\rightarrow}\underline{12}\overset{2}{\rightarrow}2.$$
    These are two non equivalent exact sequences, so global dimension is greater or equal to $1$.
    Let $A\overset{A}{\rightarrow}C_1\rightarrow ... \rightarrow C_{n}\overset{B}{\rightarrow} B$ an $n$-exact sequence. It is equivalent to the long exact sequence 
    $$A\overset{A}{\rightarrow}A\rightarrow \emptyset... \emptyset\rightarrow B\overset{B}{\rightarrow} B$$ with $C_1\overset{A}{\rightarrow}A$, $C_n\overset{B}{\rightarrow}B$ and trivial morphisms elsewhere.
    \smallbreak Now if $G$ belongs to the list in the second part of the proposition it is straightforward to verify the result.
    $\qedhere$
\end{proof}
\begin{rmk}
    In this context the usual Euler form $(A,B)=\prod_i|Ext^i(A,B)|^{(-1)^i}$ is just $\frac{|Ext^1(A,B)|}{|Hom(A,B)|}$ like in the quiver context.
\end{rmk}
\begin{prop}
    For $A,B\in X_1(G)$, $Ext^1(A,B)\simeq E(A,B)$. In particular if \newline $A\cap B\neq \emptyset$, then $Ext^1(A,B)=\emptyset$.
\end{prop}
\begin{proof}
A set of edges $e\subseteq E(A,B)$ gives an element $A\rightarrow AeB \rightarrow B$ of $Ext^1(A,B)$. Take $e$ and $e'$ two different sets such that $AeB$ and $Ae'B$ give two equivalent exact sequences by a morphism $AeB\overset{X}{\rightarrow}{Ae'B}$. The commutativity imply that $V(X)=V(A)\sqcup V(B)$, as $X$ is induced in $AeB$ we have that $X=AeB$, the same for $X=Ae'B$ and it is a contradiction. 
$\qedhere$
\end{proof}
\begin{prop}
    The following equality holds :
    $$2^{\phi(A,B)}=\frac{|Ext^1(A,B)|}{|Hom(A,B)|}.$$
\end{prop}
\begin{proof}
If $A\cap B\neq \emptyset$, then $|Hom(A,B)|\geq 1$ and $|Ext^1(A,B)|=0$ then 
$$2^{\phi(A,B)}=0=\frac{|Ext^1(A,B)|}{|Hom(A,B)|}.$$
If $A\cap B= \emptyset$ then $|Hom(A,B)|=|\{\emptyset\}|=1$ and $|Ext^1(A,B)|=2^{|E(A,B)|}$ and the equality holds again.
    $\qedhere$
\end{proof}
\subsection{The naive coproduct}
We can define another coproduct for the Hall algebra associated to an undirected labelled graph. We first sought to determine whether the natural coproduct in Hall algebra was compatible, but in some combinatorial contexts where the notion of direct sum makes sense, another coproduct is sometimes used: 
$$\Delta(f)(M,N)=f(M\oplus N).$$
It can be see in \cite{szczesny2011representationsquiversf1} for quiver representations over $\mathbb{F}_1$ or in \cite{eppolito2018protoexactcategoriesmatroidshall} for matroids. Instead of breaking a module in all the possible submodules, it is just breaking among the direct sums. In the graph context the direct sum could be analogue to the disjoint union of graphs.
\begin{definition}\label{defnaivecoproduct}
    Let $H\in X_1(G)$, we denote by $C_1,...,C_k$ the connected components of $H$. Define $$\Delta'(H)=\sum_{I\subseteq \{1,...,k\}}C_I\otimes H\backslash C_I,$$
    where $C_I=C_{i_1}\sqcup ... \sqcup C_{i_p}$ if $I=\{i_1,...,i_p\}$.
\end{definition}
\begin{ex}\label{exampleofprimitivenaivecoproduct}
    If $H$ is a connected subgraph, $\Delta'(H)=H\otimes\emptyset+\emptyset\otimes H$, i.e. $H$ is a primitive element.
    \newline If $K$ is a graph with connected components $C_1, C_2$, then $$\Delta'(K)=K\otimes \emptyset+C_1 \otimes C_2+C_2\otimes C_1+ \emptyset\otimes K.$$
\end{ex} 
\begin{prop}
    The coproduct $\Delta'$ is coassociative and cocommutative.
\end{prop}
\begin{proof}
The cocommutativity is straightforward. The coassociativity is not much more difficult. For $A\in X_1(G)$, $(\Delta'\otimes id)\circ \Delta'(A)=(id\otimes \Delta')\circ \Delta'(A)$ because is just all the manner to break $A$ into three disjoint parts.
    $\qedhere$
\end{proof}
Even for this new coproduct we must define a twisted multiplication on $\hall'(G)\otimes \hall'(G)$ in order to make it a twisted bialgebra.
\begin{definition}
    For $A,B\in X_1(G)$, we define $\phi'$ by
    \begin{equation*}
    \phi'(A,B)=
    \left\{\begin{array}{lcl}
         0 \mbox{ if }A\cap B=\emptyset,  \\
        -\infty \mbox{ otherwise.}
    \end{array}\right.
    \end{equation*}
\end{definition}
\begin{thm}\label{bigebrecoproduitnaif}
    With the same counit as before, $(\hall'(G),\cdot,\Delta',\emptyset,\epsilon)$ is a $(\Q,2,\phi',\phi')$-bialgebra.
\end{thm}
\begin{proof}
We could give a proof similar to \cref{thmbialgebraNOL}. We provide a shorter proof, although it is less thorough conceptually. \newline
Let $A,B\in X_1(G)$ with $A\cap B=\emptyset$ otherwise $\Delta'(AB)=\Delta'(A)\Delta'(B)=0$ by the twist that we introduce. \newline We begin by computing $$\Delta'(AB)=\Delta'(\sum_{e\subseteq E(A,B)}AeB)=\sum_{e\subseteq E(A,B)}\Delta'(AeB)=\sum_{e\subseteq E(A,B)}\sum \limits_{\underset{\text{s.t. }O_1\sqcup O_2=AeB}{O_1,O_2}}O_1\otimes O_2.$$
Notice that $O_1$ and $O_2$ give a disjoint decomposition of $A$ in $A_1,A_2$ and of $B$ in $B_1,B_2$.
So this factor will appear  only in the product $(A_1\otimes A_2)*(B_1\otimes B_2)$ that will be in $\Delta'(A)\Delta'(B)$. It proves that every factor of $\Delta'(AB)$ appears in $\Delta'(A)\Delta'(B)$ once. \smallbreak Now consider a factor appearing in $\Delta'(A)\Delta'(B)$, it is given by $$A_1e_1B_1\otimes A_2e_2B_2$$ with $A=A_1\sqcup A_2, B=B_1\sqcup B_2$ and $e_1\subseteq E(A_1,B_1), e_2\subseteq E(A_2,B_2).$ Now define $e=e_1\sqcup e_2$, $O_1=A_1e_1B_1$ and $O_2=A_2e_2B_2$, the factor $A_1e_1B_1\otimes A_2e_2B_2=O_1\otimes O_2$ will appear only in $\Delta'(AeB)$ once. It proves the equality $\Delta'(AB)=\Delta'(A)\Delta'(B)$.
    $\qedhere$
\end{proof}
\begin{prop}
    The set $Con(G)$ of connected nonempty subgraphs of $G$ is a basis of the vector space $Prim'(G)$ of primitive elements for $\Delta'$.
\end{prop}
\begin{proof}
Let $P$ be a primitive element of $\hall'(G)$, let's write $P=\sum_i \lambda_iA_i$ with $A_i\in X_1(G)$ with all $A_i$ distinct. We have $$\Delta'(P)=\emptyset\otimes P + P \otimes \emptyset=\sum_i\lambda_i\Delta'(A_i).$$
It follows that all the $A_i$ are primitives too because a factor from $\Delta'(A_i)$ allow us to get back all the information of $A_i$. So if two factors can combine, it is because at least two $A_i$'s are equals.

Now it remains to prove that an element of $X_1(G)$ is primitive if and only if it is a connected subgraph. By \cref{exampleofprimitivenaivecoproduct} we saw that a connected subgraph is a primitive element. If a graph is not connected, then by definition there is a non trivial term in its coproduct.
    $\qedhere$
\end{proof}
\begin{prop}
    The set $Con(G)$ generates $\hall'(G)$ as an algebra.
\end{prop}
\begin{proof}
    We just have to prove that every element $A\in X_1(G)$ can be generated by $Con(G)$. We do an induction on the number of connected components of $A$. If it is $0$ then $A$ is the empty graph and by definition is in the subalgebra generated by $Con(G)$. If it is $1$, $A$ is in $Con(G)$ by definition. If the number of connected components is $k$, let $A_1$ be a connected component and consider the product $A_1 \cdot (A\backslash A_1)$. By definition it is a sum of terms of the form $A_1e(A\backslash A_1)$ with $e\subseteq E(A_1,A\backslash A_1)$. So the sum consists of $A$ and terms that have less than $k$ connected components. Then by induction $A$ is generated by $Con(G)$ and it concludes the proof.
    $\qedhere$
\end{proof} 
\begin{rmk}
    The result of \cref{thmpresentationNOL} can be obtained in a similar way. In conclusion, we observe that two coproducts over $\hall(G)$ yield the same presentation; the primitive elements are not identical in both cases, but they are in a canonical bijection. This may seem surprising at first glance, because when the product and coproduct are adjoint, the set of primitives is minimal (in the sense defined in \cite{Berenstein_2016} theorem 1.2), which is by no means guaranteed in the case of the naive coproduct. We will see that this fact is merely a coincidence and would no longer hold in the case of directed graphs, where the number of primitives is much smaller in the case of the Hall coproduct. 

\end{rmk} 

\subsection{Polyhedral product}
In \cite{bergnerunpublishednote}, an unpublished note, Bergner, Kuhn and Zakharevich pose the following question : Is there a topological space associated to every graph whose cohomology ring reproduces the Hall algebra? Their strongest conjecture predicts that this would be the topological realization of the simplicial set $X_G$. In this section, we partially answer this question by associating with each graph a topological space whose Hall algebra is its cohomology. After a small change in degree, the space found corresponds to that given by the topological realization in the small cases, but we do not know if this holds in general. We follow \cite{bahri2008polyhedralproductfunctormethod} for the introduction to polyhedral products.
    \begin{enumerate}

\item Let $(\underline{X},\underline{x})=\{(X_i,x_i)^n_{i=1}\}$ denote
a set of $CW$--complexes with base-point $x_i$.
\item Let $K$ denote a simplicial complex with $m$ vertices labeled by the set
\newline $[n]=\{1,2,\ldots, n\}$. We assume that every $i \in [n]$ belongs to at least one simplex.
\end{enumerate}
\begin{definition}
     For every $\sigma$ simplex of $K$, let
$$
D(\sigma) =\prod^n_{i=1}Y_i,\quad {\rm where}\quad
Y_i=\left\{\begin{array}{lcl}
X_i &{\rm if} & i\in \sigma,\\
\{x_i\} &{\rm if} & i\in [n]-\sigma,
\end{array}\right.$$ with $D(\emptyset) = \{x_1\} \times \ldots \times \{x_n\}$.

The polyhedral product is defined to be
$$Z(K;(\underline{X},\underline{x}))=\bigcup_{\sigma \in K}
D(\sigma)\subseteq \prod_{i=1}^n X_i.$$ 
\end{definition}
\begin{thm}\cite{bahri2008polyhedralproductfunctormethod}(Theorem 2.35)\label{cohopoly}
    Let $K$ be a simplicial complex with $n$ vertices and let
    \begin{equation*}
        (\underline{X},\underline{x})=\{(X_i,x_i)_{i=1}^n\}
    \end{equation*}
    be a collection of $n$ pointed, connected CW-complexes. If all of the spaces $X_1, . . . , X_n$ are of finite type and have torsion-free cohomology over $\Z$, then there is an isomorphism of algebras
    \begin{equation*}
        \bigotimes_{i=1}^nH^*(X_i;\Z)/I(K)\rightarrow H^*(Z(K;(\underline{X},\underline{x}));\Z).
    \end{equation*}
    where $I(K)$ is the ideal generated by all elements $x_{j_1} \otimes x_{j_2} \otimes ... \otimes x_{j_l}$ for which $x_{j_t} \in \bar{H}^*(X_{j_t} ;\Z)$ and the sequence $J = (j_1 , . . . , j_l )$ is not a simplex of $K$, where $\bar{H}$ means the reduced cohomology.
\end{thm}
Now we apply this theorem to our context of graphs.
\begin{definition}
    Let $G$ be a graph with $H_1,...,H_n$ all the nonempty connected subgraphs of $G$. Define the simplicial complex $K_G$ as follow :
    \begin{enumerate}
        \item $K$ has $n$ vertices $1,...,n$.
        \item $\sigma=\{s_1,...,s_k\}$ is a face of $K$ if for all $i\neq j \in \sigma$, $H_i\cap H_j= \emptyset$.
    \end{enumerate}
    Also define $T_G:=(\underline{X},\underline{x})$ with $X_i=\Sph^{2|V(H_i)|}$ and $x_i$ a point in $X_i$.
\end{definition}
\begin{ex}
    Let $K$ be the graph in \cref{exmulti}, then the associated simplicial complex is
   
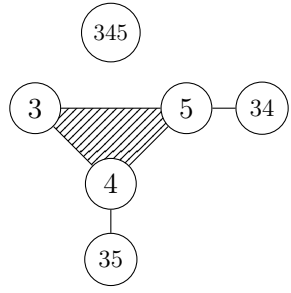
\begin{center} \begin{tikzpicture}
   \smallbreak
      \coordinate (3) at (0,0);
  \coordinate (4) at (1,-1);
  \coordinate (5) at (2,0);

  \fill[pattern=north east lines, pattern color=black] (3) -- (4) -- (5) -- cycle;
\node[fill=white,draw,circle, scale=1](3) at (0,0){3};
    \node[fill=white,draw,circle,  scale=1](4) at (1,-1){4};
    \node[fill=white,draw,circle, scale=1](5) at (2,0){5};
    \node[draw,circle, scale=0.85](34) at (3,0){34};
    \node[draw,circle, scale=0.85](35) at (1,-2){35};
     \node[draw,circle, scale=0.8](345) at (1,1){345};

    \draw (3)--(5);
    \draw (3)--(4);
    \draw (4)--(5);
    \draw (4)--(35);
    \draw (5)--(34);
\end{tikzpicture}
\end{center}
\smallbreak
and the polyhedral product is given by
\begin{equation*}
    Z(K_K;T_K)=(\Sph^2)^3\times \{*\}^3 \cup \{*\}^2\times \Sph^2\times \Sph^4 \times \{*\}^2 \cup \{*\}\times \Sph^2 \times \{*\}^2 \times \Sph^4 \times \{*\} \cup \{*\}^5 \times \Sph^6  
\end{equation*}
inside $(\Sph^2)^3\times (\Sph^4)^2\times \Sph^6$.
\end{ex}
By \cref{cohopoly} we have the following result.
\begin{thm}\label{thmpolyhedralproductNOL}
    Let $G$ be a graph, then there is an algebra isomorphism
    \begin{equation*}
        \hall(G)_\Z\simeq H^*(Z(K_G;T_G);\Z),
    \end{equation*}
    where $\hall_\Z(G)$ is the Hall algebra of $G$ with coefficients in $\Z$.
    Moreover, this isomorphism respects the grading if we double it.
\end{thm}
\begin{ex}
Let us give some examples for some classes of graphs
\leavevmode\\
\begin{center}
\begin{tabular}{|>{\centering\arraybackslash}m{3cm}|>{\centering\arraybackslash}m{6cm}|}
\hline
Graph &Polyhedral product  \\
\hline
\vspace{0.3cm}
\begin{tikzpicture}
\node[draw,circle, scale=.5](1) at (0,0){1};
\node[draw,circle, scale=.5](2) at (0.4,0){2};
\node (3) at (0.8,0){...};
\node[draw,circle, scale=.5](n) at (1.2,0){n};
\end{tikzpicture} \vspace{0.3cm} & $(\Sph^2)^n$\\
\hline
\begin{tikzpicture}
\node[draw,circle, scale=.5](1) at (0,0){1};
\draw[->] (1) edge[loop above] node[above] {\it{\tiny{k loops}}} (1);

\end{tikzpicture} & $\bigvee_{i=0}^{2^k}\Sph^2$ \\
\hline
\vspace{0.3cm}
\begin{tikzpicture}
\node[draw,circle, scale=.5](1) at (0,0){1};
\node[draw,circle, scale=.5](2) at (0.5,0){2};
\draw (1)--(2);
\end{tikzpicture}\vspace{0.3cm} & $(\Sph^2\times \Sph^2)\vee \Sph^4$ \\
\hline
$K_3$ & \vspace{0.3cm}$\begin{aligned}
    &(\Sph^2)^3\times \{*\}^7 \\
    &\cup \Sph^2 \times \{*\}^2 \times \Sph^4 \times \{*\}^6 \\
    &\cup \{*\} \times \Sph^2 \times \{*\}^2 \times \Sph^4 \times \{*\}^5 \\
    &\cup \{*\}^2 \times \Sph^2 \times \{*\}^2 \times \Sph^4 \times \{*\}^4 \\
    & \cup \{*\}^6 \times \Sph^6 \times \{*\}^3 \\
   &\cup \{*\}^7 \times \Sph^6 \times \{*\}^2 \\
   & \cup \{*\}^8 \times \Sph^6 \times \{*\}^1 \\
   & \cup \{*\}^9 \times \Sph^6 
    \end{aligned}$ \vspace{0.3cm}\\
\hline
\end{tabular}
\end{center}
As we can see, even for an easy graph like $K_3$, the associated polyhedral product could be really long but is always simple to calculate.
\end{ex}

\section{Directed graphs}\label{sectionOL}
\subsection{2-Segal sets and Hall algebras of directed graphs}
In this section we investigate the Hall algebra of directed labelled graphs. The main results are \cref{thmbialgebraOL} (twisted bialgebra), \cref{propprimitivesOL} (description of primitives) and \cref{thmpresentationOL} (presentation). A directed graph is the data of $G=(V(G),E(G))$ where $E$ is a set of tuple of vertices. One difference is that the product will no longer be commutative. We will use some notation in common with undirected graphs. In this section $G$ denotes a directed labelled graph.

We first introduce the 2-Segal set associated to a directed labelled graph.
\begin{definition}
We define the simplicial set $X(G)$ to be :
    \begin{itemize}
        \item $X_0(G)=\{ \emptyset\}$.
        \item $X_1(G)$ the set of all subgraphs of $G$.
        \item $X_n(G)=\{(S_1\subseteq...\subseteq S_n=H)~|~H\in X_1(G)$\\ \hspace{1cm}$\text{ and there is no edges from }S_{i+1}\backslash S_i \text{ to }S_i\text{ for all }i\} $.
        \smallbreak
        The face map $d_i:X_n(G) \rightarrow X_{n-1}(G)$ is given by :
        \begin{itemize}
            \item if $i=0$, $d_0(S_1\subseteq...\subseteq S_n\subseteq H)=(S_2\backslash S_1\subseteq...\subseteq S_n\backslash S_1 \subseteq H\backslash S_1)$
            \item for $i\geq 1$, $d_i$ forgets the $i^{th}$ term of the filtration.
        \end{itemize}
        \smallbreak
        The map $s_i:X_n(G)\rightarrow X_{n+1}(G)$ is just given by repeating the $i^{th}$ term.
    \end{itemize}
\end{definition}
It is left to the reader to verify that $X(G)$ is a simplicial set.
\begin{prop}
    The simplicial set $X(G)$ is 2-Segal.
\end{prop}
\begin{proof}
Let $n\geq 3$ and $0\leq i < j \leq n$.
    Take $$(S_1\subseteq ... \subseteq S_i \subseteq S_j\subseteq ... \subseteq S_n)\in X_{\{0,...,i,j,...n\}} \text{ and } (Q_{i+1}\subseteq ... \subseteq Q_j)\in X_{\{i,...j\}}$$ such that the image of this two elements in $X_{\{i,j\}}$ is equal, i.e. $S_j\backslash S_i=Q_j$. Now for every $i< k < j$ construct the graph $S_k$ with the set of vertices $V(Q_k)\sqcup V(S_i)$ and with edges such that it is a induced subgraph of $S_n$. Now we easily verify that $(S_1\subseteq ... \subseteq S_n)\in X_n$ and that this is the unique element which image is the two previous elements under the maps
    $$X_{\{0,...,n\}}\rightarrow X_{\{0,...,i,j,...,n\}} \text{ and } X_{\{0,...,n\}}\rightarrow X_{\{i,...,j\}}.$$
    
    $\qedhere$
\end{proof}
\begin{definition}
For $A$ and $B$ two disjoint elements of $X_1(G)$, we denote by $E(A,B)$ the set of edges going from $A$ to $B$.
\end{definition}
As previous we will use the following graph $G$ as a running example in this section.
 \begin{figure}[ht] \centering\begin{tikzpicture}
    \node[draw,circle, scale=1](1) at (0,0){1};
    \node[draw,circle, scale=1](2) at (1,1){2};
    \node[draw,circle, scale=1](3) at (2,0){3};
    
    \draw[->] (1)--(3);
    \draw[->] (1)--(2);
    \draw[->] (2)--(3);
    
 \end{tikzpicture} 
 \caption{A directed graph $G$}\label{exdirected}
\end{figure}
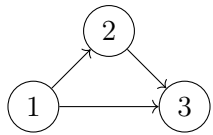
\begin{prop}
    Let $A,B\in X_1(G)$, the Hall algebra associated to $X(G)$ has the following multiplication :
    \begin{equation*} A B = \left\{
    \begin{array}{ll}
        \sum_{e \subseteq E(A,B)} Ae B ~~\mbox{ if }A\cap B=\emptyset,\\
        0 ~~\mbox{ otherwise. }
    \end{array}
\right.\end{equation*}
\end{prop}
\begin{ex}
    For the graph $G$ (\cref{exdirected}) we have
    \begin{align*}
    \vcenter{\hbox{\begin{tikzpicture}
    \node[draw,circle, scale=.5](1) at (0,0){1};
    \end{tikzpicture}}}
    \hspace{0.5cm}
    \cdot 
    \vcenter{\hbox{\begin{tikzpicture}
    \node[draw,circle, scale=.5](2) at (1,1){2};
    \node[draw,circle, scale=.5](3) at (2,0){3};
    \draw[->] (2)--(3);
    \end{tikzpicture}}}
    =
    \vcenter{\hbox{\begin{tikzpicture}
    \node[draw,circle, scale=.5](1) at (0,0){1};
    \node[draw,circle, scale=.5](2) at (1,1){2};
    \node[draw,circle, scale=.5](3) at (2,0){3};
    \draw[->] (2)--(3);
    \end{tikzpicture}}}
    + 
    \vcenter{\hbox{\begin{tikzpicture}
    \node[draw,circle, scale=.5](1) at (0,0){1};
    \node[draw,circle, scale=.5](2) at (1,1){2};
    \node[draw,circle, scale=.5](3) at (2,0){3};
    \draw[->] (1)--(2);
    \draw[->] (2)--(3);
    \end{tikzpicture}}}
    +
    \vcenter{\hbox{\begin{tikzpicture}
    \node[draw,circle, scale=.5](1) at (0,0){1};
    \node[draw,circle, scale=.5](2) at (1,1){2};
    \node[draw,circle, scale=.5](3) at (2,0){3};
    \draw[->] (1)--(3);
    \draw[->] (2)--(3);
    \end{tikzpicture}}}
    +
    \vcenter{\hbox{\begin{tikzpicture}
    \node[draw,circle, scale=.5](1) at (0,0){1};
    \node[draw,circle, scale=.5](2) at (1,1){2};
    \node[draw,circle, scale=.5](3) at (2,0){3};
    \draw[->] (1)--(3);
    \draw[->] (1)--(2);
    \draw[->] (2)--(3);
    \end{tikzpicture}}}
    \end{align*}
\end{ex}
\begin{definition}
    For $A\in X_1(G)$ define $\Gamma(A)$ to be the set of induced subgraphs $H$ of $A$ such that 
    $E(A\backslash H,H)=\emptyset$.
\end{definition}
The usual Hall coproduct (\cref{coproducthallalgebra}) is reinterpreted as follow.
\begin{prop}
    For $A\in X_1(G)$,
    $$\Delta(A)=\sum_{H\in \Gamma(A)}H\otimes A\backslash H.$$
\end{prop}
\begin{ex}
    Let's compute the coproduct of $G$ from \ref{exdirected}
    \begin{align*}
       \Delta(\vcenter{\hbox{\begin{tikzpicture}
   \node[draw,circle, scale=.5](1) at (0,0){1};
   \node[draw,circle, scale=.5](2) at (0.5,0.5){2};
    \node[draw,circle, scale=.5](3) at (1,0){3};
    \draw[->] (1)--(3);
    \draw[->] (1)--(2);
    \draw[->] (2)--(3);
    \end{tikzpicture}}})
   =
   \emptyset \otimes 
    \vcenter{\hbox{\begin{tikzpicture}
    \node[draw,circle, scale=.5](1) at (0,0){1};
    \node[draw,circle, scale=.5](2) at (0.5,0.5){2};
    \node[draw,circle, scale=.5](3) at (1,0){3};
    \draw[->] (1)--(3);
    \draw[->] (1)--(2);
    \draw[->] (2)--(3);
    \end{tikzpicture}}}
    + 
    \vcenter{\hbox{\begin{tikzpicture}
    \node[draw,circle, scale=.5](1) at (0,0){1};
    \end{tikzpicture}}}
    \otimes 
    \vcenter{\hbox{\begin{tikzpicture}
    \node[draw,circle, scale=.5](2) at (0.5,0.5){2};
     \node[draw,circle, scale=.5](3) at (1,0){3};
    \draw[->] (2)--(3);
    \end{tikzpicture}}}
    +
    \vcenter{\hbox{\begin{tikzpicture}
    \node[draw,circle, scale=.5](1) at (0,0){1};
    \node[draw,circle, scale=.5](2) at (0.5,0.5){2};
    \draw[->] (1)--(2);
    \end{tikzpicture}}}
    \otimes
     \vcenter{\hbox{\begin{tikzpicture}
    \node[draw,circle, scale=.5](3) at (1,0){3};
    \end{tikzpicture}}}
    +
    \vcenter{\hbox{\begin{tikzpicture}
    \node[draw,circle, scale=.5](1) at (0,0){1};
   \node[draw,circle, scale=.5](2) at (0.5,0.5){2};
    \node[draw,circle, scale=.5](3) at (1,0){3};
    \draw[->] (1)--(3);
    \draw[->] (1)--(2);
    \draw[->] (2)--(3);
    \end{tikzpicture}}}
    \otimes \emptyset
    \end{align*}
\end{ex}
\subsection{Twisted bialgebra and presentation}
\begin{definition}
Let's define $\tilde{\phi}$ to be 
    \begin{equation*}
\begin{array}{cccc}
    &\tilde{\phi} : V\times V& \rightarrow &\Z\sqcup \{-\infty\}   \\
     &(v_1,v_2) &\mapsto& \left\{
    \begin{array}{ll}
        -\infty ~~\mbox{ if }v_1=v_2\\
        |E(v_1,v_2)| ~~\mbox{ otherwise. }
    \end{array}\right.
\end{array}
\end{equation*}
It defines a bilinear map of monoids $\phi : \N V \times \N V \rightarrow \Z\sqcup \{-\infty\}$. The only difference  with \cref{deftwistNOL} here is that the set of edges is directed.
\end{definition}
\begin{thm}\label{thmbialgebraOL}
    Let $G$ be a directed labelled graph, then $(\hall(G),\cdot,\Delta,\emptyset,\epsilon)$ is a $(\Q,2,\phi,\phi)$-bialgebra.
\end{thm}
\begin{proof}
Let $A,B\in X_1(G)$, if $A B=0$ then $A\cap B\neq \emptyset$ and for each term $A_1\otimes A_2$ of $\Delta(A)$ and $B_1\otimes B_2$ in $\Delta(B)$ either $A_1B_1=0$, $A_2B_2=0$, $A_1\cap B_2\neq \emptyset$ or $A_2\cap B_1 \neq \emptyset$. It shows that if $\Delta(A B)=0$ then $\Delta(A)\cdot \Delta(B)=0$ too.

Now we focus on the case $AB\neq 0$. So
$$\Delta(AB)=\sum_{e\subseteq E(A,B)}\Delta(AeB)=\sum_{e\subseteq E(A,B)}\sum_{O\in \Gamma(AeB)}O\otimes AeB\backslash O.$$
As in the proof of \cref{bigebrecoproduitnaif}, $O$ gives us a partition of $A$ in $A_1$, $A_2$ and the same for $B$. So this term will appear only in the multiplication $A_1B_1\otimes A_2B_2$ exactly once. 

For the other side,
\begin{align*}
    \Delta(A)\cdot\Delta(B)=&\sum_{O_A\in \Gamma(A), O_B\in \Gamma(B)}(O_A\otimes A\backslash O_A)\cdot(O_B\otimes B\backslash O_B)
\end{align*}
it is clear that if all the four terms above are not disjoints, it cannot appear in $\Delta(AB)$, this is why the twist remove this terms. Now, a term of $\Delta(A)\cdot \Delta(B)$ is of the form 
$$A_1e_1B_1\otimes A_2e_2 B_2.$$
This term appears in $\Delta(AeB)$ where $e:=e_1\sqcup e_2$ exactly once. But for how many $e\subseteq E(A,B)$ is it the case ? Each choice $e_3\subseteq (A_1,B_2)$ and $e_4(A_2,B_1)$ gives a valid $e:=e_1\sqcup e_2 \sqcup e_3 \sqcup e_4$, and these are the only possibilities. So a term $A_1e_1B_1\otimes A_2e_2B_2$ appears exactly $2^{|E(A_1,B_2)|+|E(A_2,B_1)|}$ that is the coefficient given by the twist.
$\qedhere$
\end{proof}
\subsection{Primitive elements}
\begin{definition}
    If $G$ is a directed graph we denote by $\bar{G}$ the undirected associated graph. We say that $G$ is strongly connected if for every vertices $a$ and $b$ of $G$, there exists an oriented path from $a$ to $b$. An undirected graph is said to be strongly orientable if there exists a choice of orientation for each edge such that the resulting directed graph is strongly connected.
\end{definition}
\begin{rmk}
    An element of $X_1(G)$ is primitive if and only if it is strongly connected. Indeed to be primitive for a basis element means that there is no non trivial cut of the graph that is exactly the same as being strongly connected.
\end{rmk}
\begin{definition}
    Let $G$ be a directed graph and $H$ an induced subgraph of $G$. If $C$ is an orientation of $\bar{H}$, we denote by $H_C$ the directed graph whose vertices are $V(H)$ and whose edges are those of $C$ (it may no longer be a subgraph of $G$). Furthermore, we denote by $H_{C\cap H}$ the directed graph whose vertices are $V(H)$ and whose edges are given by $C\cap E(H)$. If $C$ is a strongly connected orientation, we say that $H_{C\cap H}$ is induced by a strongly connected orientation.
    
    For $V\subseteq V(G)$ we denote by $G_{V}$ the induced subgraph of $G$ over the vertex set $V$.
\end{definition}
\begin{ex}
    Let's consider the graph $G$ (\cref{exdirected}) and choose $C$ to be the clockwise orientation of $\bar{G}$ :
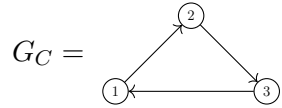
    \begin{align*}
    G_C=\vcenter{\hbox{
    \begin{tikzpicture}
    \node[draw,circle, scale=.5](1) at (0,0){1};
    \node[draw,circle, scale=.5](2) at (1,1){2};
    \node[draw,circle, scale=.5](3) at (2,0){3};
    \draw[->] (3)--(1);
    \draw[->] (1)--(2);
    \draw[->] (2)--(3);
    \end{tikzpicture}}}
    \end{align*}
    then 
    \begin{align*}
        G_{C\cap E(G)}= \vcenter
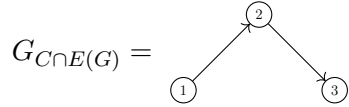
{\hbox{
    \begin{tikzpicture}
    \node[draw,circle, scale=.5](1) at (0,0){1};
    \node[draw,circle, scale=.5](2) at (1,1){2};
    \node[draw,circle, scale=.5](3) at (2,0){3};
    \draw[->] (1)--(2);
    \draw[->] (2)--(3);
    \end{tikzpicture}}}.
    \end{align*}
    If we denote by $\tilde{C}$ the counter clockwise orientation, then 
    \begin{align*}
        G_{\tilde{C}}=\vcenter
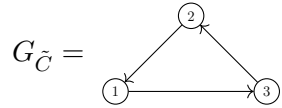
{\hbox{
    \begin{tikzpicture}
    \node[draw,circle, scale=.5](1) at (0,0){1};
    \node[draw,circle, scale=.5](2) at (1,1){2};
    \node[draw,circle, scale=.5](3) at (2,0){3};
    \draw[->] (1)--(3);
    \draw[->] (2)--(1);
    \draw[->] (3)--(2);
    \end{tikzpicture}}}
    \end{align*}
    and 
    \begin{align*}
        G_{\tilde{C}\cap E(G)}=\vcenter
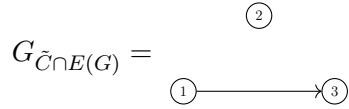
{\hbox{
    \begin{tikzpicture}
    \node[draw,circle, scale=.5](1) at (0,0){1};
    \node[draw,circle, scale=.5](2) at (1,1){2};
    \node[draw,circle, scale=.5](3) at (2,0){3};
    \draw[->] (1)--(3);
    \end{tikzpicture}}}.
    \end{align*}
\end{ex}
\begin{lem}
   Let $P$ be a directed graph and denote by $C_1,...,C_k$ the strongly connected components of $P$. For $i\neq j\in \{0,...,n\}$ there exists at least one cut $A\otimes B$ of $P$ such that $C_i$ is in $A$ and $C_j$ in $B$ or $C_j$ is in $A$ and $C_i$ in $B$.
\end{lem}
\begin{proof}
Suppose without loss of generality that there are only edges going from $C_i$ to $C_j$, we then want to find a cut $A\otimes B$ such that $C_i$ is in $A$ and $C_j$ in $B$. We give a construction step by step. At the beginning $A$ is constitute of $C_i$ and $B$ of $C_j$.  If a connected component $C_n$ has a path going to $A$ put it in $A$, if there is a path going from $B$ to $C_n$ put it in $B$. If it is not in the two previous situation put it where you want. At the end $A\otimes B$ is a regular cut of $P$, indeed, all the vertices have been considered, a strongly connected component cannot be in the two first situations at the same time otherwise there is a contradiction about the strongly connected components. It is straightforward to verify that there is no edges from $B$ to $A$. 
$\qedhere$
\end{proof}
\begin{prop}\label{prop_orientation_alors_reste_non_nul}
    A non strongly connected subgraph $P$ of $G$ is induced by a strongly connected orientation if and only if for each non trivial cut $A_1\otimes A_2$ appearing in its coproduct, $E_P(A_1,A_2)\subsetneq E_G(A_1,A_2)$, i.e., for each cut there exists at least one edge from $A_1$ to $A_2$ in $G$ that is not in $P$.
\end{prop}
\begin{proof}
Suppose that $P$ is induced by $C$ a strongly connected orientation and let $A_1\otimes A_2$ be a non trivial cut of $P$. Suppose that $E_P(A_1,A_2)= E_G(A_1,A_2)$, because $P_C$ is strongly connected there is at least one edge from $A_2$ to $A_1$, because of the previous equality this edge is chosen among edges in $G$ that are already from $A_2$ to $A_1$. So in $P$ this edge must appear because the orientation agree. So in $P$ the cut $A_1\otimes A_2$ cannot appear.
Conversly, suppose that for each non trivial cut $A_1\otimes A_2$ of $P$ there exists at least another edge in $G$ from $A_1$ to $A_2$. Construct an orientation such that each edges in $P$ is chosen with the same orientation and others edges of $G$ not in $P$ are chosen with the converse orientation. The direct result is that $(G_{V(P)})_{C\cap P}=P$. Now it remains to check that $C$ is a strongly connected orientation. Indeed using the previous lemma and according to the manner of choosing orientations, all strongly connected components of $P$ are connected in $(G_{V(P)})_C$.
$\qedhere$ 
\end{proof}
\begin{definition}
    For $P$ a subgraph of $G$ we define $S_P$ to be the element of $\hall(G)$ given by
    $$S_P:=\sum_{E(P)\subseteq E \subseteq E(G_{V(P)})} (-1)^{|E|}P_E.$$
\end{definition}
\begin{ex}
    Again for the graph $G$ (\cref{exdirected}),
    \begin{align*}
        S_{
    \begin{tikzpicture}
    \node[draw,circle, scale=.3](1) at (0,0){1};
    \node[draw,circle, scale=.3](2) at (0.3,0.3){2};
    \node[draw,circle, scale=.3](3) at (0.6,0){3};
    \draw[->] (2)--(3);
    \end{tikzpicture}}= 
    -
    \vcenter{\hbox{
    \begin{tikzpicture}
    \node[draw,circle, scale=.5](1) at (0,0){1};
    \node[draw,circle, scale=.5](2) at (0.5,0.5){2};
    \node[draw,circle, scale=.5](3) at (1,0){3};
    \draw[->] (2)--(3);
    \end{tikzpicture}}}
    +
    \vcenter{\hbox{
    \begin{tikzpicture}
    \node[draw,circle, scale=.5](1) at (0,0){1};
    \node[draw,circle, scale=.5](2) at (0.5,0.5){2};
    \node[draw,circle, scale=.5](3) at (1,0){3};
    \draw[->] (2)--(3);
    \draw[->] (1)--(3);
    \end{tikzpicture}}}
    +
    \vcenter{\hbox{
    \begin{tikzpicture}
    \node[draw,circle, scale=.5](1) at (0,0){1};
    \node[draw,circle, scale=.5](2) at (0.5,0.5){2};
    \node[draw,circle, scale=.5](3) at (1,0){3};
    \draw[->] (2)--(3);
    \draw[->] (1)--(2);
    \end{tikzpicture}}}
    -
    \vcenter{\hbox{
    \begin{tikzpicture}
    \node[draw,circle, scale=.5](1) at (0,0){1};
    \node[draw,circle, scale=.5](2) at (0.5,0.5){2};
    \node[draw,circle, scale=.5](3) at (1,0){3};
    \draw[->] (2)--(3);
    \draw[->] (1)--(3);
    \draw[->] (1)--(2);
    \end{tikzpicture}}}.
    \end{align*}
\end{ex}
\begin{prop}
    The set $\{S_P | P \in X_1(G)\}$ is a basis of the vector space $\hall(G)$.
\end{prop}
\begin{proof}
Order the elements of $X_1(G)$ by lexicographic order on the number of vertices and the number of edges and notice that in this basis the transition matrix is lower triangular with only $\pm 1$'s on the diagonal.
$\qedhere$
\end{proof}
\begin{prop}
    If $P$ is a subgraph of $G$ induced by a strongly connected orientation, then $S_P$ is a primitive element.
\end{prop}
\begin{proof}
Notice that in the sum $S_P$, every cut of an element of this sum is also a cut of the graph $P$. It is therefore sufficient to verify that each cut of P is compensated by the cuts of the rest of the sum. If $P$ is strongly connected then any element of the sum is strongly connected and so $S_P$. If not, let $A_1\otimes A_2$ be a  non trivial cut of $P$. Terms where this factor appears differ from $P$ by adding edges from $A_1$ to $A_2$. We have seen in \cref{prop_orientation_alors_reste_non_nul} that this choice is nonempty. Denote by $e(P,A_1,A_2)$ this nonempty set of edges. In conclusion
$$\langle \Delta(S_P),A_1\otimes A_2\rangle=\sum_{E\subseteq e(P,A_1,A_2)}(-1)^{|E(P)|+|E|}=\sum_{k=0}^n \binom{n}{k}(-1)^{k+|E(P)|}=0,$$
where $n:= |e(P,A_1,A_2)|>0$.
$\qedhere$
\end{proof}
\begin{ex}
   For the graph $G$ (\cref{exdirected}), there is 5 primitive elements induced by a strongly connected orientation :
   \begin{align*}
    & \begin{tikzpicture}
    \node[draw,circle, scale=.5](1) at (0,0){1};
    \end{tikzpicture}
    ,
    \begin{tikzpicture}
    \node[draw,circle, scale=.5](2) at (0,0){2};
    \end{tikzpicture}
    ,
    \begin{tikzpicture}
    \node[draw,circle, scale=.5](3) at (0,0){3};
    \end{tikzpicture}
    ,\\
    &S_{\begin{tikzpicture}
    \node[draw,circle, scale=.3](1) at (0,0){1};
    \node[draw,circle, scale=.3](2) at (0.3,0.3){2};
    \node[draw,circle, scale=.3](3) at (0.6,0){3};
    \draw[->] (2)--(3);
    \draw[->] (1)--(2);
    \end{tikzpicture}}=\vcenter{\hbox{
    \begin{tikzpicture}
    \node[draw,circle, scale=.5](1) at (0,0){1};
    \node[draw,circle, scale=.5](2) at (0.5,0.5){2};
    \node[draw,circle, scale=.5](3) at (1,0){3};
    \draw[->] (2)--(3);
    \draw[->] (1)--(2);
    \end{tikzpicture}}}
    -
    \vcenter{\hbox{
    \begin{tikzpicture}
    \node[draw,circle, scale=.5](1) at (0,0){1};
    \node[draw,circle, scale=.5](2) at (0.5,0.5){2};
    \node[draw,circle, scale=.5](3) at (1,0){3};
    \draw[->] (2)--(3);
    \draw[->] (1)--(2);
    \draw[->] (1)--(3);
    \end{tikzpicture}}}
    ,\\
    &S_{\begin{tikzpicture}
    \node[draw,circle, scale=.3](1) at (0,0){1};
    \node[draw,circle, scale=.3](2) at (0.3,0.3){2};
    \node[draw,circle, scale=.3](3) at (0.6,0){3};
    \draw[->] (1)--(3);
    \end{tikzpicture}}=-\vcenter{\hbox{
    \begin{tikzpicture}
    \node[draw,circle, scale=.5](1) at (0,0){1};
    \node[draw,circle, scale=.5](2) at (0.5,0.5){2};
    \node[draw,circle, scale=.5](3) at (1,0){3};
    \draw[->] (1)--(3);
    \end{tikzpicture}}}
    +
    \vcenter{\hbox{
    \begin{tikzpicture}
    \node[draw,circle, scale=.5](1) at (0,0){1};
    \node[draw,circle, scale=.5](2) at (0.5,0.5){2};
    \node[draw,circle, scale=.5](3) at (1,0){3};
    \draw[->] (2)--(3);
    \draw[->] (1)--(3);
    \end{tikzpicture}}}
    +
    \vcenter{\hbox{
    \begin{tikzpicture}
    \node[draw,circle, scale=.5](1) at (0,0){1};
    \node[draw,circle, scale=.5](2) at (0.5,0.5){2};
    \node[draw,circle, scale=.5](3) at (1,0){3};
    \draw[->] (1)--(3);
    \draw[->] (1)--(2);
    \end{tikzpicture}}}
    -
    \vcenter{\hbox{
    \begin{tikzpicture}
    \node[draw,circle, scale=.5](1) at (0,0){1};
    \node[draw,circle, scale=.5](2) at (0.5,0.5){2};
    \node[draw,circle, scale=.5](3) at (1,0){3};
    \draw[->] (2)--(3);
    \draw[->] (1)--(3);
    \draw[->] (1)--(2);
    \end{tikzpicture}}}.
   \end{align*}
\end{ex}
\begin{prop}\label{propprimitivesOL}
    The set $\{S_P | P \text{ induced by a strongly connected orientation }\}$ is a basis of the vector space $Prim(G)$.
\end{prop}
\begin{proof}
Let $K$ be a primitive element of $\hall(G)$. Express $K$ in the $S_P$ basis. 
$$K=\sum_{P_k \text{ induced by SCO}} \lambda_k S_{P_k} + \sum_{Q_k \text{ not induced by a SCO}} \mu_k S_{Q_k},$$
where SCO means strongly connected orientation, $\lambda_k$ and $\mu_k$ are non zero and all $P_k$ and $Q_k$ are different. As $K$ and the first sum are primitive it implies that $\sum_{Q_k} \mu_k S_{Q_k}$ is primitive too. Among all the $Q_k$'s take the one with lower number of vertices and among graph with this same minimal number of vertices choose one with a minimal number of edges, denote it by $Q_0$. As $Q_0$ is not induced by a strongly connected orientation by the \cref{prop_orientation_alors_reste_non_nul} there exists a cut $A_1\otimes A_2$ of $Q_0$ such that all edges from $A_1$ to $A_2$ in $G$ are in $Q_0$. The term $A_1\otimes A_2$ will appear in the coproduct of $Q_0$ and cannot be compensated by any other term of the coproduct of any $S_{Q_k}$ as there is no other edges in $G$ and $Q_0$ has a minimal number of edges. This proves that the sum of $S_{Q_k}$'s cannot be primitive and so all $\mu_k$ are zero. 
$\qedhere$
\end{proof}

\begin{rmk}
    $\hall(G)$ is $\N$-graded by the number of vertices (after forgetting the labels), so it is generated by primitive elements (\cref{corgenerate}).
\end{rmk}
\begin{cor}\label{cordimensionprimitives}
The dimension of primitives in each degree depends only on the undirected graph $\bar{G}$.
\end{cor}
\begin{proof}
The dimension of $Prim(G)_I$ where $I\in \N V$ only depends on the different orientations of subgraphs of $\bar{G}$ with vertices corresponding to $I$.
$\qedhere$
\end{proof}
Now, using the twisted bialgebra structure and the primitive elements previously determined we can give a presentation of the Hall algebra of $G$.
\begin{thm}\label{thmpresentationOL}
    We have the following isomorphism :
    \begin{equation*}
        \hall(G)\simeq \Q\langle x_P~|~P \text{ induced by a strongly connected orientation}\rangle /R
    \end{equation*}
    where $R$ are the following relations :
    \begin{enumerate}
        \item $x_{P_1}...x_{P_k}=0$ if there exists $i \neq j$ such that $V(P_i)\cap V(P_j)\neq \emptyset$.
        \item $x_Px_Q=x_Qx_P$ if there are no edges in $G$ between $P$ and $Q$ (both directions).
    \end{enumerate}
    \end{thm}
\begin{proof}
This isomorphism maps $x_P$ to $S_P$. It is easy to verify that it is well defined. By \cref{corgenerate} it is surjective. Now it remains to prove that every relation between primitive elements are induced by relations of $R$. By \cref{proprelationprimitifs} we only considered relation where products are permutation of the same terms. Let $\sum_\sigma a_\sigma S_{P_{\sigma(1)}}...S_{P_{\sigma(k)}}$ be a relation, we already suppose that all set of vertices are disjoint otherwise by the first relation of $R$, it is zero. Among all the products in this sum, take a subgraph $K$ appearing in these products with the maximal amount of edges. As the sum is zero, it means that $K$ appears in at least two products, we denote by $S_1...S_k$ and $S_1'...S_k'$ these two products. Suppose that we maximise the correspondence of this two products from the left to the right using the commutativity relation of $R$. Denote by $j$ the first index where the terms of the two products are different ($j$ is well defined since the order of terms are different). 
We have the following :
\begin{equation*}
S_1...S_{j-1}S_j...S_l...S_k=
    S_1'...S'_{j-1}S_l...S_j...S_k'
\end{equation*}
Now in the second expression, try to pass $S_j$ to the left. If $S_j$ arrive just after $S'_{j-1}$ it is a contradiction with the hypothesis of maximal correspondence. So there is an integer $m$, such that on the way from $S_j$ to $S_{j-1}$ there is edges between $S_m$ and $S_j$ in $G$. If there is an edge from $S_m$ to $S_j$ it is a contradiction of the fact that $K$ appears in $S_1...S_k$ by the maximal number of edges of $K$. If there is an edge from $S_j$ to $S_m$ it is a contradiction of the fact that $K$ has the maximal number of edges. Therefore, there are no other relations besides those in $R$.
$\qedhere$
\end{proof}
\begin{rmk}
    Notice that the presentation of the Hall algebra does not depend on the orientation of the graph.
\end{rmk}
\begin{ex}
 In the following expressions, the letter $X$ and $Y$ represent any product of variables.
\begin{enumerate}
    \item For a graph $G$ with $n$ vertices and no edges, we have
    \begin{equation*}
        \hall(G)\simeq \Q[ x_1,x_2,...,x_n]/(x_i^2 \text{ for }i\in \{1,...,n\}).
    \end{equation*}
    \item For the graph $K_2$
    \begin{equation*}
        \hall(G)\simeq \Q\langle x_1,x_2 \rangle / (x_iXx_i \text{ for }i\in \{1,2\}).
    \end{equation*}
    \item For the graph $K_3$
    \begin{align*}
        \hall(G)\simeq \Q\langle x_1,x_2,x_3,x_4,x_5 \rangle / (x_i&Xx_i \text{ for }i\in \{1,2,3\},Xx_iY \text{ for }i\in \{4,5\} \\ &\text{ with } X \text{ or }Y \text{ non empty}).
    \end{align*}
    \item If $G$ is a tree with $n$ vertices
    \begin{align*}
    \hall(G)\simeq \Q\langle x_1,x_2,..,x_n \rangle / (x_i&Xx_i \text{ for }i\in \{1,...,n\}, x_ix_j=x_jx_i \text{ if there are}\\ &\text{ no edges between }i \text{ and }j).
    \end{align*}
\end{enumerate}
\end{ex}
\begin{rmk}\label{rmknaivecoproduct}
    If we consider the naive coproduct as in \cref{defnaivecoproduct}
    for undirected graphs, it is easy to see that the primitives are the directed subgraphs whose underlying undirected graph is connected. We thus see that in this case the number of primitives is much larger, and we no longer obtain results as precise as those in this section.
\end{rmk}
\subsection{The Hall polynomial of a graph}
\begin{definition}
    For $G$ a directed graph we denote by $P^m_G$ the Poincaré polynomial of the graded vector space of primitives :
    $$P^m_G(t)=\sum_{I\in \N V}^n \text{dim}(Prim(G)_I)t_I,$$
    where $t_I$ is defined to be $\prod t_i^{k_i}$ if $I=\sum k_i i\in \N V$. We will refer to this polynomial as the multivariate Hall polynomial of the graph $G$.
\end{definition}
As a direct consequence of \cref{cordimensionprimitives} we have
\begin{cor}
 The polynomial $P^m_G$ is an invariant of the underlying undirected graph.
\end{cor}
\begin{rmk}
    For $I\subseteq V(G)$ we see it in $\N V$ by $I= \sum_{i\in I}i$. As every graduation not of this form is zero in $\hall(G)$ it's easy to see that every non zero coefficient of $P^m_G$ appears for a $I\in \N V$ of this form.
\end{rmk}
\begin{ex}
    \begin{enumerate}
        \item If $\bar{G}$ is a tree on $n$ vertices, then $P^m_G=\sum_{i\in V}t_i$.
        \item If $\bar{G}$ is $K_3$ the complete graph on $3$ vertices, then $P^m_G=t_1+t_2+t_3+2t_1t_2t_3$.
        \item If $\bar{G}$ is the graph $K_4$, then $P^m_G=t_1+t_2+t_3+t_4+2t_1t_2t_3+2t_1t_2t_4+2t_2t_3t_4+2t_1t_3t_4+24t_1t_2t_3t_4$.
        \item In general if $\bar{G}=K_n$, then $P^m_G=\sum_{I\subseteq V}c_{|I|}t_I,$ where $(c_k)_{k\in \N}$ is the sequence counting the number of strongly connected labeled tournaments on $k$ nodes (OEIS A054946). 
    \end{enumerate}
        \end{ex}
Since this polynomial counts strongly connected orientations (\cref{propprimitivesOL}), it seems reasonable to draw a connection to the Tutte polynomial. Indeed, the evaluation of the Tutte polynomial of a connected graph in $(0,2)$ counts the number of strongly connected orientations of that graph. This property can be found in any good survey on the Tutte polynomial ; for example, \cite{tuttepolynomial}.
\begin{prop}
Let $G$ be a directed graph such that $\bar{G}$ is connected. Then the coefficient of $t_{V(G)}$ is equal to $T_{\bar{G}}(0,2)$. More generally, for $I\subseteq V(G)$ the coefficient of $t_I$ is equal to $T_{\bar{G_I}}(O,2)$, where $G_I$ denotes the induced subgraph on the vertices of $I$.
\end{prop}
\begin{proof}
This is a direct application of \cref{propprimitivesOL}.
$\qedhere$
\end{proof}
\begin{prop}
Let $G$ and $H$ be two directed graphs. We have the following identities 
$$P^m_{G\sqcup H}=P^m_{G-H}=P^m_G+P^m_H$$
where $G-H$ denotes the graphs $G$ and $H$ connected by a single edge.
\end{prop}
\begin{proof}
It's clear there is no strong orientation for an induced subgraph with a non empty part in $G$ and $H$ as this graph is not connected. If $G$ and $H$ are connected by a single edge it is the same argument as this edge is a bridge (i.e. an edge that disconnects the graph if it is removed) and a graph with a bridge can not have a strong orientation.
$\qedhere$
\end{proof}
We now define a specialization of the multivariate Hall polynomial to get a polynomial that forgets the labels.
\begin{definition}
    For $G$ a directed graph we define the Hall polynomial of $G$ denoted by $P_G$ as the specialization of $P^m_G$ in $t_i=t$ for all $i\in V$.
\end{definition}
\begin{rmk}
This polynomial inherits all the properties of the multivariate Hall polynomial, namely
    \begin{itemize}
        \item The polynomial $P_G$ is an invariant of the underlying graph $\bar{G}$.
        \item For $H$ and $G$ two directed graphs we have $P_{G\sqcup H}=P_{G-H}=P_G+P_H$.
        \item The coefficient of $t^k$ in $P_G$ counts the number of strongly connected orientations on induced subgraphs with $k$ vertices.
    \end{itemize}
\end{rmk}
\begin{ex}
    \begin{enumerate}
        \item If $\bar{G}$ is a tree on $n$ vertices, then $P_G=nt$.
        \item If $\bar{G}$ is the complete graph $K_n$ on $n$ vertices, then 
        $$P_G=\sum_{i=0}^n\binom{n}{k}c_k t^k.$$
        
        In particular $P_{K_2}=2t$, $P_{K_3}=3t+2t^3$, $P_{K_4}=4t+8t^3+24t^4$.
        \item Consider the graphs $G_1$ and $G_2$ of \cref{exgraphssametutte}. These are two graphs that share the same Tutte polynomial 
        \begin{align*}
            x^6 + 5x^5 &+ 5x^4y + 5x^3y^2 + 5x^2y^3 + 3xy^4 + y^5 + 10x^4 + 15x^3y + 14x^2y^2 + 9xy^3 + 3y^4 \\&+ 10x^3 + 16x^2y + 12xy^2 + 4y^3 + 5x^2 + 7xy + 3y^2 + x + y
        \end{align*}
        but have different Hall polynomials :
        $$P_{G_1}(t)=7t^3+10t^3+30t^4+90t^5+138t^6+126t^7$$
        and
        $$P_{G_2}(t)=7t^3+10t^3+30t^4+86t^5+138t^6+126t^7.$$
        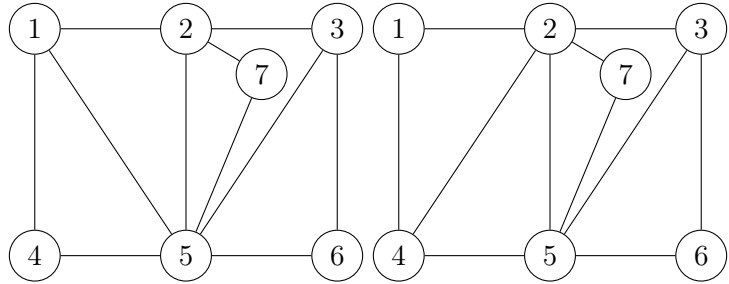
\begin{figure}[H] \centering\begin{tikzpicture}
    \node[draw,circle, scale=1](4) at (0,0){4};
    \node[draw,circle, scale=1](5) at (2,0){5};
    \node[draw,circle, scale=1](6) at (4,0){6};
    \node[draw,circle, scale=1](1) at (0,3){1};
    \node[draw,circle, scale=1](2) at (2,3){2};
    \node[draw,circle, scale=1](3) at (4,3){3};
    \node[draw,circle, scale=1](7) at (3,2.4){7};
    \draw (1)--(2);
    \draw (1)--(4);
    \draw (1)--(5);
    \draw (2)--(3);
    \draw (2)--(5);
    \draw (2)--(7);
    \draw (3)--(5);
    \draw (3)--(6);
    \draw (4)--(5);
    \draw (5)--(6);
    \draw (5)--(7);
 \end{tikzpicture} 
 \begin{tikzpicture}
    \node[draw,circle, scale=1](4) at (0,0){4};
    \node[draw,circle, scale=1](5) at (2,0){5};
    \node[draw,circle, scale=1](6) at (4,0){6};
    \node[draw,circle, scale=1](1) at (0,3){1};
    \node[draw,circle, scale=1](2) at (2,3){2};
    \node[draw,circle, scale=1](3) at (4,3){3};
    \node[draw,circle, scale=1](7) at (3,2.4){7};
    \draw (1)--(2);
    \draw (1)--(4);
    \draw (2)--(4);
    \draw (2)--(3);
    \draw (2)--(5);
    \draw (2)--(7);
    \draw (3)--(5);
    \draw (3)--(6);
    \draw (4)--(5);
    \draw (5)--(6);
    \draw (5)--(7);
 \end{tikzpicture} 
 \caption{Two graphs $G_1$ (on the left) and $G_2$ (on the right).} \label{exgraphssametutte}
\end{figure}
        \item Flower graphs are a good class of examples to consider if we want to found graphs with the same Tutte polynomial but different Hall polynomial. See \cite{flowergraphs} for more information on these graphs. Flower graphs of same size and same multiset of petals have the same Tutte polynomial but it's not true for the Hall polynomials in general. See \cref{examplesflowergraphs} for some computations. Here we define a flower graph by the length of its central cycle and the sizes of its petals in clockwise order (the order doesn't matter as the a flower graph of size $n$ has the diedral group $D_n$ as automorphism group).
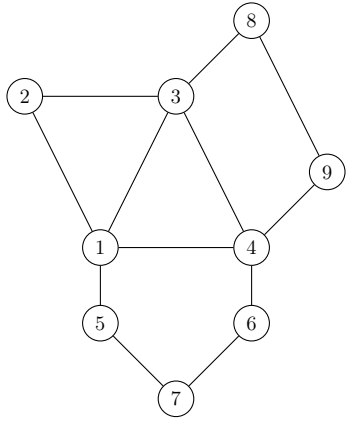
\begin{figure}[H] \centering\begin{tikzpicture}
    \node[draw,circle, scale=0.7](1) at (0,0){1};
    \node[draw,circle, scale=0.7](2) at (-1,2){2};
    \node[draw,circle, scale=0.7](3) at (1,2){3};
    \node[draw,circle, scale=0.7](4) at (2,0){4};
    \node[draw,circle, scale=0.7](5) at (0,-1){5};
    \node[draw,circle, scale=0.7](6) at (2,-1){6};
    \node[draw,circle, scale=0.7](7) at (1,-2){7};
     \node[draw,circle, scale=0.7](8) at (2,3){8};
      \node[draw,circle, scale=0.7](9) at (3,1){9};
    \draw (1)--(3);
    \draw (1)--(4);
    \draw (3)--(4);
    \draw (1)--(2);
    \draw (2)--(3);
    \draw (3)--(8);
    \draw (8)--(9);
    \draw (9)--(4);
    \draw (4)--(6);
    \draw (6)--(7);
    \draw (5)--(7);
    \draw (5)--(1);
 \end{tikzpicture} 
 \caption{The flower graph of size $3$ with petals $[3,4,5]$.} \label{exflowergraph}
\end{figure}        
\begin{figure}[H]
\begin{center}
\begin{tabular}{|>{\centering\arraybackslash}m{2cm}|>{\centering\arraybackslash}m{3cm}|>{\centering\arraybackslash}m{8cm}|}
\hline
Size of the flower & Sizes of petals in order & Hall polynomial \\
\hline
$3$ & $[3,0,3]$ & $5t+6t^3+12t^4+18t^5$ \\
\hline 
$3$ & $[3,4,5]$ & $9t+4t^3+8t^4+8t^5+ 24t^6+ 18t^7+ 18t^8+ 46t^9$ \\
\hline
$4$ & $[3,3,4,4]$ & $10t+4t^3+6t^4+16t^5+38t^6+76t^7+126t^8+108t^9+146t^{10}$ \\
\hline 
$4$ & $[3,4,3,4]$ & $10t +4t^3 +6t^4 +12t^5 +46t^6 +72t^7+ 126t^8+ 108t^9+ 146t^{10}$ \\
\hline 
\end{tabular}
\caption{Some examples of Hall polynomials of flower graphs} \label{examplesflowergraphs}
\end{center}
\end{figure}

    \end{enumerate}
    
\end{ex}
\begin{rmk}
    Notice that the last two graphs have the same Tutte polynomial but different Hall polynomials. As a last example, the graphs $(n,[3,3,0,...,0])$ and $(n,[3,0,3,0,...,0])$ have the same Tutte polynomial for each $n$ but different Hall polynomials for $n\geq 4$. For the moment, we are not aware of any examples of graphs with the same Hall polynomial but different Tutte polynomials.
\end{rmk}
\bibliographystyle{amsalpha}
\bibliography{biblio.bib}

\end{document}